\documentclass[12pt]{amsart}
\usepackage{geometry}
\usepackage{xcolor}
\usepackage{url}
\usepackage{amsthm}
\usepackage{amstext}
\usepackage{amssymb}
\usepackage[all]{xy}
\PassOptionsToPackage{normalem}{ulem}
\usepackage{ulem}

\makeatletter

\numberwithin{equation}{section}
\numberwithin{figure}{section}
\theoremstyle{plain}
\newtheorem{thm}{Theorem}[section]
  \theoremstyle{remark}
  \newtheorem{rem}[thm]{Remark}
  \theoremstyle{plain}
  \newtheorem{prop}[thm]{Proposition}
  \theoremstyle{plain}
  \newtheorem{lem}[thm]{Lemma}
  \theoremstyle{plain}
  \newtheorem{cor}[thm]{Corollary}
  \theoremstyle{plain}
  \newtheorem{conjecture}[thm]{Conjecture}
  \theoremstyle{definition}
  \newtheorem{defn}[thm]{Definition}
  \theoremstyle{definition}
  \newtheorem{problem}[thm]{Problem}

\makeatletter

\newcommand{\quot}{/\!\!/}

\def\hom{\mathsf{Hom}}

\def\quot{/\!\!/}
\renewcommand{\hom}{\mathsf{Hom}}

\newcommand{\C}{\mathbb{C}}
\newcommand{\X}{\mathfrak{X}}
\newcommand{\R}{\mathbb{R}}
\newcommand{\Z}{\mathbb{Z}}
\newcommand{\Q}{\mathbb{Q}}
\newcommand{\GL}{\mathsf{GL}}
\newcommand{\SL}{\mathsf{SL}}
\newcommand{\U}{\mathsf{U}}
\newcommand{\SU}{\mathsf{SU}}

\title[Character Varieties of Abelian Groups]{Topology of character varieties of Abelian groups}

\author[C. Florentino]{Carlos Florentino}

\address{Departamento Matem\'atica, Instituto Superior T\'ecnico, Av. Rovisco
Pais, 1049-001 Lisbon, Portugal}

\email{cfloren@math.ist.utl.pt}

\author[S. Lawton]{Sean Lawton}

\address{Department of Mathematics, The University of Texas-Pan American,
1201 West University Drive Edinburg, TX 78539, USA}

\email{lawtonsd@utpa.edu}

\thanks{This work was partially supported by the project GEAR (NSF-RNMS 1107367) USA, Simon's Foundation Collaboration grant (\#245642) USA, NSF grant 1309376 USA, and the projects PTDC/MAT/099275/2008 and PTDC/MAT/120411/2010, FCT, Portugal.}



\makeatother

\begin{document}
\begin{abstract}
Let $G$ be a complex reductive algebraic group (not necessarily connected),
let $K$ be a maximal compact subgroup, and let $\Gamma$ be a finitely
generated Abelian group. We prove that the conjugation orbit space
$\hom(\Gamma,K)/K$ is a strong deformation retract of the GIT quotient
space $\hom(\Gamma,G)\quot G$. Moreover, this result remains true
when $G$ is replaced by its locus of real points. As a corollary,
we determine necessary and sufficient conditions for the character
variety $\hom(\Gamma,G)\quot G$ to be irreducible when $G$ is connected
and semisimple. For a general connected reductive $G$, analogous
conditions are found to be sufficient for irreducibility, when $\Gamma$
is free Abelian.
\end{abstract}
\maketitle

\section{Introduction}

The description of the space of commuting elements in a compact Lie
group is an interesting algebro-geometric problem with applications
in mathematical physics, remarkably in supersymmetric Yang-Mills theory
and mirror symmetry (\cite{BFM,KaSm,Th,Wi}). Some special cases of
this problem and related questions have recently received attention,
as can be seen, for example, in the articles \cite{AGC,BaJeSe,Ba1,FlLa,GoPeSo,PeSo,Si5}.

Let $K$ be a compact Lie group and view the $\mathbb{Z}$-module
$\mathbb{Z}^{r}$ as a free Abelian group of rank $r$, for a fixed
integer $r>0$. The space of commuting $r$-tuples of elements in
$K$ can be naturally identified with the set $\hom(\mathbb{Z}^{r},K)$
of group homomorphisms from $\mathbb{Z}^{r}$ to $K$, by evaluating
a homomorphism on a set of free generators for $\mathbb{Z}^{r}$.
From both representation-theoretic and geometric viewpoints, one is
interested in homomorphisms up to conjugacy, so the quotient space
$\X_{\Gamma}(K):=\hom(\mathbb{Z}^{r},K)/K$, where $K$ acts by conjugation,
is the main object to consider. Since every compact Lie group is isomorphic
to a matrix group, it is not difficult to see that this orbit space
is a semialgebraic space, but many of its general properties remain
unknown.

In this article, we also consider the analogous space for a complex
reductive affine algebraic group $G$, or its real points $G(\R)$.
More generally, in many of our results we replace the free Abelian
group $\mathbb{Z}^{r}$ by an arbitrary finitely generated Abelian
group $\Gamma$.

In this context, it is useful to work with the geometric invariant
theory (GIT) quotient space, denoted by $\hom(\Gamma,G)\quot G$,
and usually called the $G$-character variety of $\Gamma$ (see Section
2 for details). The character varieties $\X_{\Gamma}(G):=\hom(\Gamma,G)\quot G$
are naturally affine algebraic sets; not necessary irreducible, smooth,
nor homotopically trivial. In the case of its real locus $G(\R)$,
we consider the space of closed orbits, also denoted $\X_{\Gamma}(G(\R))$,
which is semi-algebraic by \cite{RS}.

Here is our first main result: 
\begin{thm}
\label{thm:Main}Let $G$ a complex or real reductive algebraic group,
$K$ a maximal compact subgroup of $G$, and $\Gamma$ a finitely
generated Abelian group. Then there exists a strong deformation retraction
from $\X_{\Gamma}(G)$ to $\X_{\Gamma}(K)$. 
\end{thm}
We remark that in \cite{FlLa}, the analogous result for complex reductive
$G$ was shown for $\Gamma$ a free (non-Abelian) group $\mathsf{F}_{r}$
of rank $r$; that is, the free group character variety $\hom(\mathsf{F}_{r},G)\quot G$
deformation retracts to $\hom(\mathsf{F}_{r},K)/K$. In the same article,
the special case $K=\mathsf{U}(n)$, $G=\mathsf{GL}(n,\C)$ and $\Gamma=\mathbb{Z}^{r}$
of Theorem \ref{thm:Main} was also shown. Likewise this result remains
true for $G(\R)$ and $\Gamma=\mathsf{F}_{r}$, see \cite{CFLO}.

Pettet and Souto in \cite{PeSo} have shown that, under the hypothesis
of Theorem \ref{thm:Main}, $\hom(\mathbb{Z}^{r},G)$ deformation
retracts to $\hom(\mathbb{Z}^{r},K)$. An intermediate space along
their retraction is the space of representations with closed orbits,
which is also important for understanding the topology of the GIT
quotient. Thus, some steps in their result are useful for our independent
proof of Theorem \ref{thm:Main}, which is moreover very different
from our argument in \cite{FlLa}.

In fact, there are great differences between the free Abelian and
free (non-Abelian) group cases. For instance, the deformation retraction
from $\hom(\mathsf{F}_{r},G)\cong G^{r}$ to $\hom(\mathsf{F}_{r},K)\cong K^{r}$
is trivial, although the deformation from $\X_{\mathsf{F}_{r}}(G)=\hom(\mathsf{F}_{r},G)\quot G$
to $\X_{\mathsf{F}_{r}}(K)=\hom(\mathsf{F}_{r},K)/K$ is not. Moreover,
in the case of free Abelian groups the deformation is not determined
in general by the factor-wise deformation (as happens in the free
group case), since any such deformation cannot preserve commutativity
(see \cite{So}). We conjecture that the analogous result is valid
for right angled Artin groups, a class of groups that interpolates
between free and free Abelian ones (see Section \ref{sub:RAAG}).

Returning to the situation of a general finitely generated group $\Gamma$,
the interest in the spaces $\hom(\Gamma,K)/K$ and $\hom(\Gamma,G)\quot G$
is also related to their differential-geometric interpretation. Consider
a differentiable manifold $B$ whose fundamental group $\pi_{1}(B)$
is isomorphic to $\Gamma$ (when $\Gamma=\mathbb{Z}^{r}$, we can
choose $B$ to be an $r$-dimensional torus). By fixing a base point
in $B$ and using the standard holonomy construction in the differential
geometry of principal bundles, one can interpret $\hom(\Gamma,K)$
as the space of pointed flat connections on principal $K$-bundles
over $B$, and $\X_{\Gamma}(K)=\hom(\Gamma,K)/K$ as the moduli space
of flat connections on principal $K$-bundles over $B$ (see \cite{Th}).

The use of differential and algebro-geometric methods to study the
geometry and topology of these spaces was achieved with great success
when $B$ is a closed surface of genus $g>1$ (in this case $\pi_{1}(B)$
is non-Abelian), via the celebrated Narasimhan-Seshadri theorem and
its generalizations, which deal also with non-compact Lie groups (see,
for example \cite{AB,Hi,NS,Sim1,Sim2}). Indeed, the character varieties
$\X_{\Gamma}(G)$ introduced above can be interpreted as a moduli
space of polystable $G$-Higgs bundles over a compact Kähler manifold
with $\Gamma$ the fundamental group of the manifold, or central extension
thereof (which yields an identification in the topological category,
but not in the algebraic or complex analytic ones). Some of the multiple
conclusions from this approach was the determination of the number
of components for many spaces of the form $\hom(\pi_{1}(B),H)\quot H$
for a closed surface $B$ and real reductive (not necessarily compact
or complex) Lie group $H$ (see for example \cite{BGG,GM}). We remark
that these character varieties are also main players in mirror symmetry
and the geometric Langlands programme (see, for example \cite{GGM,HaTh,KW}).

In contrast, for the case of $\pi_{1}(B)=\mathbb{Z}^{r}$, the number
of path components of $\hom(\mathbb{Z}^{r},K)/K$ has only recently
found a satisfactory answer (for general compact semisimple $K$ and
arbitrary $r$), see \cite{KaSm}.

Using this determination, and also a result of A. Sikora \cite{Si5},
one of the main applications of Theorem \ref{thm:Main} is the following
theorem. By the classification of finitely generated Abelian groups,
any such group $\Gamma$ can be written as $\Gamma=\mathbb{Z}^{r}\oplus\Gamma'$
where $r$ is called the rank of $\Gamma$ and $\Gamma'$ is the finite
group of torsion elements (all of those having finite order). We say
that $\Gamma$ is free if $\Gamma'$ is trivial. 
\begin{thm}
\label{classification} Let $G$ be a semisimple connected algebraic
group over $\C$, and let $r$ be the rank of the Abelian group $\Gamma$.
Then $\X_{\Gamma}(G)$ is an irreducible variety if and only if:

$(1)$ $\Gamma$ is free, and

$(2)$ Either $r=1$, or $r=2$ and $G$ is simply connected, or $r\geq3$
and $G$ is a product of $\SL(n,\C)$'s and $\mathsf{Sp}(n,\C)$'s. 
\end{thm}
An interesting consequence of this result is a sufficient condition
for irreducibility of $\X_{\mathbb{Z}^{r}}(G)$ for a general connected
reductive $G$ (see Corollary \ref{cor:reductive}), generalizing
the result in \cite[Prop. 2.6]{Si5}.

Theorem \ref{classification} is also related to an analogous problem
in a different context. Let $C_{r,n}$ be the algebraic set of commuting
$r$-tuples of $n\times n$ complex matrices. Determining the irreducibility
of $C_{r,n}$ is a surprisingly difficult linear algebra problem,
related to the determination of canonical forms for similarity classes
of general $r$-tuples of matrices (see \cite{GePo}). In fact, as
a consequence of deep theorems by Gerstenhaber and Guralnick (see
\cite{Ger,Gur}) $C_{r,n}$ was shown to be reducible when $r,n\geq4$,
and when $r=3$ and $n\geq32$; moreover, $C_{r,n}$ is irreducible
when $r=1,2$ (for all $n$) and when $r=3$ and $n\leq8$. The remaining
cases: $r=3$ and $n$ strictly between 8 and 32, are still open (as
far as we know). We explore the relation between this problem and the
problem of irreducibility of $\hom(\Z^{r},G)$, when $G=\GL(n,\C)$ and 
$G=\SL(n,\C)$.

We finish the Introduction with a summary of the article. The proof
of Theorem \ref{thm:Main}, focused on the case when $G$ is complex,
is divided into three main steps. The first step, carried out in Section
2, consists in obtaining the identifications: \begin{align*}
\hom(\Gamma,K)/K & =G(\hom(\Gamma,K))/G,\\
\hom(\Gamma,G)\quot G & =\hom(\Gamma,G)^{ps}/G.\end{align*}
 Here, for a subset $X\subset\hom(\Gamma,G)$, we let $G(X):=\{gxg^{-1}:\ g\in G,\ x\in X\}$,
and the superscript $ps$ refers to the subset of representations
with closed orbits (called polystable). These identifications hold,
in fact, for any finitely generated group $\Gamma$. For the second
step, in Section 3, we restrict to a fixed Abelian $\Gamma$. It consists
in showing that one can replace the polystable representations by
{}``representations'' into $G_{ss}$, the semisimple part of $G$,
and that we have $G(\hom(\Gamma,G))=\hom(\Gamma,G(K))$, where $G(K):=\{gkg^{-1}:\ g\in G,\ k\in K\}$.

Finally in Section 4, the proof is completed by constructing a strong
deformation retraction from $G_{ss}$ to $G(K)$ with a certain number
of desired properties. This last part of the proof is inspired by
the methods and results of \cite{PeSo}, although is self-contained.
Note also that (except for the partial results on right angled Artin
groups) we do not need to use their generalized Jordan decomposition,
since in our GIT quotient framework, we can work directly with polystable
representations. The case when $G$ is real is addressed in remarks
following the above outline.

In Section 5, besides proving Theorem \ref{classification}, we also
study two different characterizations of a special component of $\X_{\mathbb{Z}^{r}}(G)$,
usually called the identity component. This, together with known results
on the compact group case, provides a final application of Theorem
\ref{thm:Main} to the simple connectivity and to the cohomology ring
of the character varieties $\X_{\mathbb{Z}^{r}}(G)$, in a few examples,
such as when $G$ is $\SL(n,\C)$ or $\mathsf{Sp}(n,\C)$.


\section{Quotients and Character Varieties}

An affine algebraic group is called \emph{reductive} if it does not
contain any non-trivial closed connected unipotent normal subgroup
(see \cite{Bo,hum} for generalities on algebraic groups). Since we
do not consider Abelian varieties in this paper, we will abbreviate
the term \textit{affine algebraic group} to simply \textit{algebraic
group.}

Let $G$ be a complex reductive algebraic group (we include the possibility
that $G$ is disconnected), and let $X$ be a complex affine $G$-variety,
that is, we have an action $G\times X\to X$, $(g,x)\mapsto g\cdot x$
which is a morphism between affine varieties. We use the term \textit{affine
variety} to mean an affine algebraic set, not necessarily irreducible.
Unless indicated otherwise, all of our algebraic groups and varieties
will be considered over $\mathbb{C}$. In particular, the group $G$
is also a Lie group.

This action induces a natural action of $G$ on the ring $\mathbb{C}[X]$
of regular functions on $X$. The subring of $G$-invariant functions
$\mathbb{C}[X]^{G}$ is finitely generated since $G$ is reductive,
and defines the affine GIT (geometric invariant theory) quotient as
the corresponding affine variety, usually denoted as: \begin{equation}
X\quot G:=\mathrm{Spec}_{max}(\C[X]^{G}).\label{eq:GIT-quot}\end{equation}
 See \cite{Do} or \cite{Muk} for the details of these constructions
and an introduction to GIT. If $x\in X$, we denote by $[x]$ or by
$G\cdot x$ its $G$-orbit in $X$, and by $[\![x]\!]$ its \emph{extended
orbit} in $X$ which is, by definition, the union of the orbits whose
closure intersects $\overline{[x]}=\overline{G\cdot x}$. As shown
in \cite{Do}, $X\quot G$ is exactly the space of classes $[\![x]\!]$.
In the context of GIT, one usually considers the Zariski topology.
However, being interested also in the usual topology on Lie groups,
we will always equip our variety $X$ with a natural embedding into
an affine space $\mathbb{C}^{N}$. This induces on $X$ a natural
Euclidean topology. When we need to distinguish the two kinds of topological
closures, we use a label in the overline; we will mean the Euclidean
topology, when no explicitly reference is made.

We observe that every orbit $[x]=G\cdot x$ is the image of the action
map, which is algebraic, so $[x]$ is constructible and thus contains
a Zariski open subset $O\subset[x]$, which is dense in its Zariski
closure $\overline{[x]}^{Z}$. Hence, $O$ is also open and dense,
in the Euclidean topology, inside $\overline{[x]}^{Z}$. Therefore,
the Euclidean and Zariski closures of orbits coincide: $\overline{[x]}^{E}=\overline{[x]}^{Z}.$

\subsection{The polystable quotient\label{sub:The-polystable-quotient}}

If the orbit of $x$ is closed, $[x]=\overline{[x]}$, we say that
$x$ is \emph{polystable}. Denote the subset of polystable points
by $X^{ps}\subset X$. Since for any $g\in G$, $[g\cdot x]=[x]$,
it is clear that $X^{ps}$ is a $G$-space. It is not, in general,
an affine variety.

We first show that, considering Euclidean topologies, there is a natural
homeomorphism between the \emph{polystable quotient $X^{ps}/G$ }and
the GIT quotient $X\quot G$.

Consider the canonical projection $\pi_{ps}:X^{ps}\to X^{ps}/G$,
and the GIT projection $\pi_{G}:X\to X\quot G$. Define also $\mathcal{I}_{ps}:X^{ps}/G\hookrightarrow X\quot G$
to be the canonical map which sends a $G$-orbit $[x]$ to its extended
$G$-orbit $[\![x]\!]$. These maps form the following diagram, whose
commutativity is clear: \begin{eqnarray}
X^{ps} & \hookrightarrow & X\nonumber \\
\pi_{ps}\downarrow &  & \downarrow\pi_{G}\label{eq:polystable}\\
X^{ps}/G & \stackrel{\mathcal{I}_{ps}}{\longrightarrow} & X\quot G.\nonumber \end{eqnarray}
In \cite{Lu1,Lu2}, $\pi_{G}$ is shown to be a closed mapping on
$G$-invariant Euclidean closed sets (see also \cite{Sch1}, page
141). The following result is implicit in the work of Luna (see \cite{Lu3}).
We provide a proof for completeness, as it will be used later.
\begin{thm}
\label{pro:polystable} Let $G$ be a complex reductive algebraic
group, and $X$ a complex affine $G$-variety $($with the natural
Euclidean topology coming from $X\subset\mathbb{C}^{N})$. Then $\mathcal{I}_{ps}$
is a homeomorphism. \end{thm}
\begin{proof}
With respect to the Euclidean topologies, and the induced topologies
on the quotients, the map $\mathcal{I}_{ps}$ is continuous. This
follows from the identification of $\mathcal{I}_{ps}$ with the composition
of continuous maps \[
\xymatrix{X^{ps}/G\ar[r]^{\iota_{ps}}\ar@/_{2pc}/[rr]_{\mathcal{I}_{ps}} & X/G\ar[r]^{\pi_{\quot}} & X\quot G}
.\]
 We will show that $\mathcal{I}_{ps}$ is a bijection and its inverse
is continuous. It is a standard fact (see \cite{Do,Muk}) that every
extended orbit equivalence class contains a unique closed orbit, and
this closed orbit lies in the closure of all of the others in that
class. In particular, for every $[\![x]\!]\in X\quot G$ there exists
a representative $x^{ps}\in X$ whose orbit is closed and $x^{ps}\in[\![x]\!]$.
So, $x^{ps}\in X^{ps}$ and $[x^{ps}]\in X^{ps}/G$. Thus, $\mathcal{I}_{ps}([x^{ps}])=[\![x]\!]$
which shows $\mathcal{I}_{ps}$ is surjective.

Since two orbits $[x_{1}]$ and $[x_{2}]$ are contained in $[\![x]\!]$
if and only if $\overline{[x_{1}]}\cap\overline{[x_{2}]}\not=\emptyset$,
no two distinct closed orbits are identified in the GIT quotient $X\quot G$.
Thus $\mathcal{I}_{ps}$ is injective.

Finally, let us show that $\mathcal{I}_{ps}$ is a closed map. Suppose
that $C\subset X^{ps}/G$ is a closed set. Then, by definition $C':=\pi_{ps}^{-1}(C)\subset X^{ps}$
is $G$-invariant. By commutativity of the diagram (\ref{eq:polystable})
we have\[
\mathcal{I}_{ps}(C)=\pi_{G}(\pi_{ps}^{-1}(C))=\pi_{G}(C').\]
 Let $C''$ be the Euclidean closure of $C'$ in $X$. By definition
of subspace topologies, since $C'$ is closed in $X^{ps}$, we have
$C'=C''\cap X^{ps}$.

Now let us show that $C''$, the set of limit points of sequences
in $C'$, is $G$-invariant. Let $x_{0}\in C''$. If $x_{0}$ in $C'$,
then clearly $g\cdot x_{0}$ is in $C'\subset C''$. Otherwise, there is a sequence
$\{x_{n}\}$ in $C'$ converging to $x_{0}$. Since $G$ acts by polynomials,
the action is continuous in the Euclidean topology. Therefore, $g\cdot x_{n}$
limits to $g\cdot x_{0},$ where $g\cdot x_{n}$ is in $C'$ for all
$n$ and $g$ since $C'$ is $G$-invariant. Therefore, $g\cdot x_{0}$
is in the limit set and hence in $C''$.

Obviously, $\pi_{G}(C')\subset\pi_{G}(C'')$. Let us show the reverse
inclusion. If $[\![x]\!]\in\pi_{G}(C'')\in X\quot G$ is an extended
orbit for some $x\in C''$, then, as above, there is a $x^{ps}\in X^{ps}$
whose orbit is closed and $x^{ps}\in[\![x]\!]$. In particular, $[\![x^{ps}]\!]=[\![x]\!]$.
This implies $\overline{[x]}\cap\overline{[x^{ps}]}=\overline{[x]}\cap G\cdot x^{ps}\neq\emptyset$,
so there exists $g\in G$ such that $g\cdot x^{ps}\in\overline{[x]}$.
Thus, since $C''$ is $G$-invariant and (Euclidean) closed, $g^{-1}\cdot(g\cdot x^{ps})=x^{ps}\in C''$.
Note that here we are using the fact, observed above, that the Euclidean
and Zariski closures of orbits coincide. Therefore $x^{ps}\in C'=C''\cap X^{ps}$,
and so $[\![x]\!]=[\![x^{ps}]\!]\in\pi_{G}(C')$. We conclude that
indeed $\pi_{G}(C')=\pi_{G}(C'')$.

Since $C''$ is closed, $G$-invariant in $X$ and $\pi_{G}$ is a
closed map for $G$-invariant sets as mentioned above, we conclude
that $\mathcal{I}_{ps}(C)=\pi_{G}(C')=\pi_{G}(C'')$ is a closed set.
This shows that $\mathcal{I}_{ps}$ is a closed map. Being bijective,
it is also an open map. So, the inverse of $\mathcal{I}_{ps}$ is
continuous, and hence $\mathcal{I}_{ps}$ is a homeomorphism. \end{proof}
\begin{rem}
Since $X\quot G$ is a complete metric space, the same holds for $X^{ps}/G$
as well. The metric can be explicitly given as follows. Let $\{f_{1},...,f_{N}\}$
generate the ring of invariants $\C[X]^{G}$, define $F=(f_{1},...,f_{N})$
to be the mapping $X\to\C^{N}$, and let $||\cdot||$ be the Euclidean
metric on $X$. For $[v],[w]\in X/G$ define $d([v],[w])=||F(v)-F(w)||$.
Thus $d$ is well-defined since $F$ is $G$-invariant; and $d$ is
non-negative, symmetric and satisfies the triangle inequality because
$||\cdot||$ does. It is not definite however. Since $d([v],[w])=0$
if the Zariski closure of $[v]$ and $[w]$ intersect (even if they
are not equal orbits), this problem is exactly fixed upon restricting
to $X^{ps}/G$. 
\end{rem}

\subsection{Polystable and compact quotients for character varieties}

Let $\Gamma$ be a finitely generated group, $H$ be a Lie group,
and consider the \emph{representation space} $\hom(\Gamma,H)$, the
space of homomorphisms from $\Gamma$ to $H$. For example, when $\Gamma=\mathsf{F}_{r}$
is a free group on $r$ generators, the evaluation of a representation
on a set of free generators provides a homeomorphism with the Cartesian
product $\hom(\mathsf{F}_{r},H)\cong H^{r}$, where we consider the
compact-open topology on $\hom(\mathsf{F}_{r},H)$ with $\mathsf{F}_{r}$
given the discrete topology.

In general, by choosing generators $\gamma_{1},\cdots,\gamma_{r}$
for $\Gamma$ (for some $r$), we have a natural epimorphism $\mathsf{F}_{r}\twoheadrightarrow\Gamma$.
This allows one to embed $\hom(\Gamma,H)\subset\hom(\mathsf{F}_{r},H)\cong H^{r}$
and consider on $\hom(\Gamma,H)$ the Euclidean topology induced from
the manifold $H^{r}$.

The Lie group $H$ acts on $\hom(\Gamma,H)$ by conjugation of representations;
that is, $h\cdot\rho=h\rho h^{-1}$ for $h\in H$ and $\rho\in\hom(\Gamma,H)$.
Recall that $\hom(\Gamma,H)^{ps}$ denotes the subset of $\hom(\Gamma,H)$
consisting of representations whose $H$-orbit is closed. 
Let $\X_{\Gamma}(H)$ denote the space of closed orbits; that is,
$\hom(\Gamma,H)^{ps}/H$. Although it is not necessarily an algebraic
set, we will refer to this space as the \emph{$H$-character variety
of $\Gamma$}. 

Three main classes of examples are important for us, for each the
induced topology on $\X_{\Gamma}(H)$ will be Hausdorff. When $H=K$
is a compact Lie group, this is the usual orbit space (which is semi-algebraic
and compact) since all such orbits are closed, so $\X_{\Gamma}(K)=\hom(\Gamma,K)/K$.
When $H=G$ is a reductive algebraic group over $\C$, each orbit
closure has a unique closed sub-orbit and thus corresponds to an extended
$G$-orbit in the sense defined just before Subsection \ref{sub:The-polystable-quotient}(see
for instance \cite{Muk}). So, applying the map $\mathcal{I}_{ps}$
from Theorem \ref{pro:polystable}, the quotient $\X_{\Gamma}(G)$
is identified with the GIT quotient (see \eqref{eq:GIT-quot}): \[
\hom(\Gamma,G)\quot G:=\mathrm{Spec}_{max}(\C[\hom(\Gamma,G)]^{G}).\]
 We record this for future use as a proposition.
\begin{prop}
\label{polystablelemma}Let $G$ be a complex reductive algebraic
group, and $\Gamma$ be a finitely generated group. Then, the natural
map $\mathcal{I}_{ps}:\hom(\Gamma,G)^{ps}/G\to\hom(\Gamma,G)\quot G$
is a homeomorphism. 
\end{prop}
Lastly, in the case that $H=G(\R)$, the real points of $G$, then
by \cite{RS}, the polystable quotient is a semi-algebraic space.
To distinguish it from the complex case, we will denote the maximal
compact subgroup of $G(\R)$ by $K(\R)$.

In all cases, the spaces $\X_{\Gamma}(H)$ will be semi-algebraic
sets and thus CW-complexes in the natural Euclidean topologies we
consider in this paper.

Let now $K$ be a fixed maximal compact subgroup of a complex reductive
algebraic group $G$. Then $G$ is the Zariski closure of $K$. Over
$\C$, all reductive algebraic groups arise as the Zariski closure
of compact Lie groups . In particular, as discussed in \cite{Sch1},
we may assume $K\subset\mathsf{O}(n,\R)$ is a real affine variety
by the Peter-Weyl theorem, and thus $G\subset\mathsf{O}(n,\C)$ is
the complex points of the real variety $K$.

Using the fact that $G$ is also isomorphic to the unique complexification
$K_{\mathbb{C}}$ of $K$, one can show the following.
\begin{prop}
\label{kmodsubsetgmod} Let $\Gamma$ be a finitely generated group
and $G=K_{\mathbb{C}}$. Then the inclusion mapping $\hom(\Gamma,K)\hookrightarrow\hom(\Gamma,G)$
induces an injective mapping \[
\iota:\hom(\Gamma,K)/K\hookrightarrow\hom(\Gamma,G)\quot G\]
 such that $\iota(\hom(\Gamma,K)/K)$ is a CW-subcomplex of $\hom(\Gamma,G)\quot G$. \end{prop}
\begin{proof}
Since $G$ is the complex points of the real variety $K$, the real
points of $G$ coincide with $K$. In the same way, the set of real
points of the affine variety $\hom(\Gamma,G)\subset G^{r}$ is precisely
$\hom(\Gamma,K)$. Since $\hom(\Gamma,K)$ is compact and stable under
$K$, the result is a direct consequence of \cite[Thm. 4.3]{FlLa3}. 
\end{proof}
Given a representation $\rho\in\hom(\Gamma,G)$, the subset $\rho(\Gamma)=\{\rho(\gamma),\ \gamma\in\Gamma\}\subset G$
is the group algebraically generated by $\rho(\gamma_{1}),\cdots,\rho(\gamma_{r})$,
where $\gamma_{1},\cdots,\gamma_{r}$ are the generators of $\Gamma$.
\begin{lem}
\label{lem:K-in-ps}Let $\Gamma$ be a finitely generated group and
$G=K_{\mathbb{C}}$, Then $\hom(\Gamma,K)\subset\hom(\Gamma,G)^{ps}$.\end{lem}
\begin{proof}
Suppose $\rho\in\hom(\Gamma,K)$. Then the Euclidean closure of $\rho(\Gamma)$
in $K$ is a compact subgroup $J$ of $K$. This implies that the
Zariski closure of $\rho(\Gamma)$ coincides with the Zariski closure
of $J$ in $G$ and hence it is a reductive algebraic group (precisely
equal to the complexification of $J$). Thus, $\overline{\rho(\Gamma)}^{Z}$
is a linearly reductive group (all its linear representations are
completely reducible). In \cite[Thm. 3.6]{Ri88}, it is shown that
$\overline{\rho(\Gamma)}^{Z}$ is linearly reductive if and only if
the $G$-orbit of $\rho$ in $\hom(\Gamma,G)$ is closed. We conclude
$G\cdot\rho$ is closed, that is, $\rho\in\hom(\Gamma,G)^{ps}$. We
note that Richardson's result in \cite{Ri88} is stated for $r$-tuples
of elements in $G$. The case of closed $G$-invariant subsets of
$G^{r}$, such as our $\hom(\Gamma,G)$, is an easy consequence (see
also \cite{CaFl}). 
\end{proof}
Define the mapping $\iota_{K}$ as the composition \[
\hom(\Gamma,K)/K\hookrightarrow\hom(\Gamma,G)^{ps}/K\twoheadrightarrow\hom(\Gamma,G)^{ps}/G.\]

We want now to describe the image of $\iota_{K}$ .
\begin{prop}
\label{diagramprop} The following diagram is commutative: \[
\xymatrix{ & \hom(\Gamma,G)/G\ar@{->>}[dr]^{\pi_{\mathrm{\quot}}}\\
\hom(\Gamma,G)^{ps}/G\ar@{^{(}->}[ur]^{\iota_{ps}}\ar[rr]^{\mathcal{I}_{ps}} &  & \hom(\Gamma,G)\quot G\\
 & \hom(\Gamma,K)/K\ar@{^{(}->}[ul]_{\iota_{K}}\ar@{^{(}->}[ur]^{\iota}}
\]
 Consequently, $\iota_{K}(\X_{\Gamma}(K))\cong\iota(\X_{\Gamma}(K))$
as CW complexes. \end{prop}
\begin{proof}
Since all mappings in the diagram are composites of natural inclusions
and projections, they are continuous. The top triangle of maps is
commutative by definition of $\mathcal{I}_{ps}$. Note that $\iota$
is the cellular inclusion from Proposition \ref{kmodsubsetgmod}.
Then Proposition \ref{kmodsubsetgmod} implies that all $G$-equivalent
$K$-valued representations are in fact $K$-equivalent (else $\iota$
would not be injective). Therefore, since $\hom(\Gamma,K)\subset\hom(\Gamma,G)^{ps}$,
we conclude that $\iota_{K}$ is also injective. Therefore, the bottom
triangle of maps is commutative. 
\end{proof}
Now define \[
G(\hom(\Gamma,K)):=\{g\rho g^{-1}\ |\ g\in G\text{ and}\ \rho\in\hom(\Gamma,K)\}.\]
 Clearly, $G(\hom(\Gamma,K))\subset\hom(\Gamma,G)^{ps}$ as a $G$-subspace
since all $K$-representations have closed orbits by Lemma \ref{lem:K-in-ps},
and conjugates of representations with closed $G$-orbits likewise
have closed $G$-orbits (since $G\cdot(g\rho g^{-1})=G\cdot\rho$).
\begin{prop}
\label{gkembeds} $\iota_{K}:\hom(\Gamma,K)/K\to\hom(\Gamma,G)^{ps}/G$
is an embedding and $\iota_{K}\left(\hom(\Gamma,K)/K\right)=G(\hom(\Gamma,K))/G$. \end{prop}
\begin{proof}
Proposition \ref{diagramprop} implies that $\iota_{K}$ is a continuous
injection. We now show it is onto $G(\hom(\Gamma,K))/G$. First, let
$[\rho]_{K}\in\X_{\Gamma}(K)$. For any $\rho'\in\iota_{K}([\rho]_{K})$,
$\rho'$ is a $G$-conjugate of a $K$-conjugate of $\rho$ and so
$\rho'=g(k\rho k^{-1})g^{-1}\in G(\hom(\Gamma,K))$; thus $\iota_{K}([\rho]_{K})\in G(\hom(\Gamma,K))/G$.
Conversely, let $[\rho]_{G}\in G(\hom(\Gamma,K))/G$. Then by definition,
there exists $g\in G$ such that $g\rho g^{-1}\in\hom(\Gamma,K)$.
Thus $[g\rho g^{-1}]_{K}\in\X_{\Gamma}(K)$ and $\iota_{K}([g\rho g^{-1}]_{K})=[g\rho g^{-1}]_{G}=[\rho]_{G}$;
we conclude $G(\hom(\Gamma,K))/G\subset\iota_{K}(\X_{\Gamma}(K))$.
Therefore, $\iota_{K}(\X_{\Gamma}(K))=G(\hom(\Gamma,K))/G.$ From
Proposition \ref{polystablelemma}, $\hom(\Gamma,G)^{ps}/G\cong\hom(\Gamma,G)\quot G$
and so is Hausdorff (the latter is an algebraic subset of some $\mathbb{C}^{N}$,
with the Euclidean subspace topology). On the other hand, $\X_{\Gamma}(K)$
is compact (a closed subset of the compact $K^{N}$ is a compact set,
and a compact quotient of a compact space is compact). Since a continuous
injection from a compact space to a Hausdorff space is an embedding,
we are done. \end{proof}
\begin{rem}
\label{realcase:step1} With respect to Proposition \ref{diagramprop},
the mapping $\iota_{K}$ remains continuous when $K$ is replaced
by $K(\R)$ since that property does not depend on $\C$, it additionally
remains injective since the proof of \cite[Remark 4.7]{FlLa3} replies
only on the polar decomposition of $G$ in terms of $K$ which remains
valid for $G(\R)$ and $K(\R)$. Therefore, the proof of Proposition
\ref{gkembeds} remains true if $G$ is replaced by $G(\R)$ and $K$
is replaced by $K(\R)$. 
\end{rem}
We record the following fact for later use. 
\begin{prop}
\label{gproduct} $\X_{\Gamma}(G\times H)\cong\X_{\Gamma}(G)\times\X_{\Gamma}(H)$
for any reductive algebraic $\C$-groups $G,H$. \end{prop}
\begin{proof}
Write $\rho\in\hom(\Gamma,G\times H)$ as $\rho=(\rho_{G},\rho_{H})$.
Clearly, $\hom(\Gamma,G\times H)\cong\hom(\Gamma,G)\times\hom(\Gamma,H)$
and under this identification, the $G\times H$-action separates into
an independent $G$-action on $\hom(\Gamma,G)$ and $H$-action on
$\hom(\Gamma,H)$. Thus, as orbit spaces $\hom(\Gamma,G\times H)/(G\times H)\cong\hom(\Gamma,G)/G\times\hom(\Gamma,H)/H$.
Moreover, since the GIT quotients are determined by orbit closures,
we conclude our result simply by noting that $\overline{(G\times H)\cdot\rho}=\overline{(G\cdot\rho_{G})\times(H\cdot\rho_{H})}=\overline{G\cdot\rho_{G}}\times\overline{H\cdot\rho_{H}}.$ 
\end{proof}

\section{Finitely generated Abelian groups}

From now on, we let $\Gamma$ be a finitely generated Abelian group.
By the classification of finitely generated Abelian groups, there
are integers $s,t\geq0$ and $n_{1},\cdots,n_{t}>1$ such that \[
\Gamma\cong\Z^{s}\oplus\bigoplus_{i=1}^{t}\Z_{n_{i}},\]
 where $\mathbb{Z}_{m}$ denotes the cyclic group $\mathbb{Z}/m\mathbb{Z}$.
In general, $\hom(\Gamma,G)$ is an algebraic subvariety of $G^{s+t}$,
given by $\rho\mapsto(\rho(\gamma_{1}),...,\rho(\gamma_{s+t}))$ where
$\{\gamma_{1},\cdots,\gamma_{s}\}$ generate $\mathbb{Z}^{s}$ and
$\gamma_{s+j}$ is a generator of $\mathbb{Z}_{n_{j}}$ for $j=1,\cdots,t$.
Recall that an element in $G$ is called \textit{semisimple} if for
any finite dimensional rational representation of $G$ the element
$g$ acts completely reducibly. Let $G_{ss}$ denote the set of semisimple
elements in $G$. Then define \[
\hom(\Gamma,G_{ss})=\{\rho\in\hom(\Gamma,G)\ |\ \rho(\gamma_{j})\in G_{ss},\quad j=1,\cdots,s+t\}.\]
 This is an abuse of notation (since $G_{ss}$ is not a group) but
a harmless one, in view of the next result. Since $G_{ss}$ is preserved
by conjugation, $G$ acts on $\hom(\Gamma,G_{ss})$ by simultaneous
conjugation. In what follows we will often abbreviate $r=s+t$. Recall
that a diagonalizable group is an algebraic group isomorphic to a
closed subgroup of a torus (see \cite{Bo}).
\begin{prop}
\label{semisimplelemma} Let $\Gamma$ be a finitely generated Abelian
group and let $H_{\rho}=\overline{\rho(\Gamma)}^{Z}$ be the Zariski
closure of $\rho(\Gamma)\subset G$. Then, the following are equivalent:

$(1)$ $\rho\in\hom(\Gamma,G_{ss})$

$(2)$ $\rho\in\hom(\Gamma,G)^{ps}$

$(3)$ $H_{\rho}$ is a diagonalizable group

$(4)$ $H_{\rho}$ is a reductive group

In particular, $\hom(\Gamma,G)^{ps}/G=\hom(\Gamma,G_{ss})/G$. \end{prop}
\begin{proof}
In \cite{Ri88}, it is shown that the Zariski closure of $\rho(\Gamma)$
in $G$ is linearly reductive if and only if the $G$-orbit of $\rho$
in $\hom(\Gamma,G)\subset G^{r}$ is closed. Since being reductive
is equivalent to linearly reductive (in characteristic 0), this shows
the equivalence between (2) and (4) (which is in fact valid for any
$\Gamma$ not necessarily Abelian).

Now let $\gamma_{1},\cdots,\gamma_{r}$ be a generating set for $\Gamma$
and let $a_{j}:=\rho(\gamma_{j})\in G$, $j=1,\cdots,r$ for a fixed
$\rho\in\hom(\Gamma,G)$. Since $\Gamma$ is Abelian, $\rho(\Gamma)\subset G$,
the subgroup generated by $\{a_{1},...,a_{r}\}$, is an Abelian subgroup
of $G$. So, the Zariski closure $H_{\rho}$ of $\rho(\Gamma)$ is
an Abelian algebraic group (since commutation relations are polynomial).
Now suppose that $\rho\in\hom(\Gamma,G_{ss})$ which means by definition
that $a_{j}\in G_{ss}$, and consider a linear embedding of $G$,
$\psi:G\to\GL(n,\C)$. Then, the matrices $\psi(a_{1}),\cdots,\psi(a_{j})$
can be simultaneously conjugated by an element in $\GL(n,\C)$ to
lie in some maximal torus $T$ of $\GL(n,\C)$. Because $T$ is Zariski
closed in $\GL(n,\C)$, this means that for every $g\in H_{\rho}$
we will have $\psi(g)\in T$. Recall that the multiplicative Jordan
decomposition is preserved by homomorphisms: for $g=g_{s}g_{u}\in G$
with $g_{s}\in G_{ss}$ and $g_{u}$ unipotent, we have $\psi(g_{s})=(\psi(g))_{s}$
and $\psi(g_{u})=(\psi(g))_{u}$. Thus, for $g\in H_{\rho}$, we have
$\psi(g)=\psi(g)_{s}\cdot\psi(g)_{u}\in T$ which implies $\psi(g)_{u}=\psi(g_{u})=\mathbf{1}$,
and since $\psi$ is injective $g_{u}=\mathbf{1}$. So, $H_{\rho}$
consists only of semisimple elements of $G$ and by \cite{Bo} this
means that $H_{\rho}$ is a diagonalizable group, and hence reductive.
Thus, (1) implies (3) and so (4). Conversely, let $H_{\rho}$ be reductive.
Since it is also Abelian (as $\Gamma$ is Abelian) then, again by
\cite{Bo} it consists of semisimple elements. In particular, $\rho(\gamma_{j})\in G_{ss}$.
So (4) implies (1) as well.\end{proof}
\begin{rem}
\label{sikora} The fact that $\Gamma$ is Abelian is crucial in Proposition
\ref{semisimplelemma}. Indeed, if $\Gamma$ is not Abelian, neither
of the inclusions $\hom(\Gamma,G_{ss})\subset\hom(\Gamma,G)^{ps}$
or $\hom(\Gamma,G_{ss})\supset\hom(\Gamma,G)^{ps}$ is true in general.
Here are simple counter-examples. Let $\Gamma=\mathsf{F}_{2}$, the
free group of rank $2$, and $G=\SL(2,\mathbb{C})$. Let $\rho=(g,h)\in\hom(\Gamma,G)$
with $g=\left(\begin{array}{cc}
x & 0\\
0 & x^{-1}\end{array}\right)$ and $h=\left(\begin{array}{cc}
y & 1\\
0 & y^{-1}\end{array}\right)$, where $x,y\in\mathbb{C}\setminus\{0\}$ are each not $\pm1$. Then
both $g$ and $h$ are semi-simple, so $\rho\in\hom(\Gamma,G_{ss})$.
However, by conjugating with $\left(\begin{array}{cc}
t & 0\\
0 & t^{-1}\end{array}\right)$ and taking a limit as $t$ goes to $0$ we see that the tuple $(g,h)$
limits to $\left(\left(\begin{array}{cc}
x & 0\\
0 & x^{-1}\end{array}\right),\left(\begin{array}{cc}
y & 0\\
0 & y^{-1}\end{array}\right)\right)$, which is not in the $G$-orbit of $(g,h)$. So, $\rho\notin\hom(\Gamma,G)^{ps}$.
Now suppose that $\Gamma=\mathsf{F}_{3}$, the free group on 3 generators
and let $\rho=(g,h,k)\in G^{3}$ with $G=\SL(2,\mathbb{C})$ again.
Consider \[
g=\left(\begin{array}{cc}
1 & 1\\
0 & 1\end{array}\right),\quad h=\left(\begin{array}{cc}
x & 0\\
0 & x^{-1}\end{array}\right),\quad k=\left(\begin{array}{cc}
y & z\\
z & w\end{array}\right)\]
 satisfying $yw-z^{2}=1$ and $(x-x^{-1})z\neq0$. If $\rho$ is an
irreducible representation, then its orbit is closed; and if it were
reducible, then $(h,k)$ could be made simultaneously upper-triangular.
However, a simple computation shows this to be impossible. Thus, $\rho\in\hom(\mathsf{F}_{3},G)^{ps}$,
but $g\notin G_{ss}$, so $\rho\notin\hom(\mathsf{F}_{3},G_{ss})$. 
\end{rem}
\ 
\begin{rem}
\label{realcase:step2} A \emph{real reductive algebraic group} is
exactly the real locus of a complex reductive algebraic group $G$,
with respect to an embedding into some $\C^{N}$. It can be arranged
for this locus $G(\R)$ to be the fixed point set in $G$ by complex
conjugation. By definition, the semisimple elements of $G(\R)$ is
those that are semisimple in $G$; that is, $G(\R)_{ss}=G_{ss}\cap G(\R)$.
Items (1) and (2) from Proposition \ref{semisimplelemma} remain equivalent
for $G(\R)$. This follows since by definition $\hom(\Gamma,G(\R)_{ss})\subset\hom(\Gamma,G_{ss})$,
and thus by Proposition \ref{semisimplelemma} $\rho\in\hom(\Gamma,G(\R)_{ss})$
if and only if the $G$-orbit of $\rho$ is closed. However, by \cite{Bi71},
the $G$-orbit of $\rho$ is closed if and only if the $G(\R)$-orbit
of $\rho$ is closed. Therefore, as in the complex case, we conclude
that $\hom(\Gamma,G(\R)_{ss})=\hom(\Gamma,G(\R))^{ps}$ and consequently
$\hom(\Gamma,G(\R))^{ps}/G(\R)=\hom(\Gamma,G(\R)_{ss})/G(\R)$ since
both are $G(\R)$-spaces. 
\end{rem}
Next, define ${\displaystyle G(K)=\bigcup_{g\in G}gKg^{-1}}$ to be
the set of all $G$-conjugates of the group $K$, and consequently
define \[
\hom(\Gamma,G(K))=\{\rho\in\hom(\Gamma,G)\ |\ \rho(\gamma_{1}),...,\rho(\gamma_{r})\in G(K)\}.\]
 Clearly, $G(\hom(\Gamma,K))\subset\hom(\Gamma,G(K))$. 
\begin{lem}
\label{ghomkhomgk} If $\Gamma$ is a finitely generated Abelian group,
then \[
G(\hom(\Gamma,K))=\hom(\Gamma,G(K)).\]
 \end{lem}
\begin{proof}
Since $G(\hom(\Gamma,K))\subset\hom(\Gamma,G(K))$ for any $\Gamma$,
it suffices to prove the reverse inclusion for finitely generated
Abelian groups. This argument follows the one in \cite[pg. 16]{PeSo}.
Let $\Gamma$ be generated by $\{\gamma_{1},...,\gamma_{r}\}$ and
let $\rho\in\hom(\Gamma,G(K))$. Let $A$ be the Zariski closure of
the group generated by $\{\rho(\gamma_{1}),...,\rho(\gamma_{r})\}$.
Since $G(K)$ only consists of semisimple elements, the proof of Lemma
\ref{semisimplelemma} tells us that $A$ is Abelian and consists
of only semisimple elements. Since $A$ is algebraic it has a maximal
compact subgroup $B$. Since it is Abelian, $B$ is unique. Since
each $\rho(\gamma_{i})$ is conjugate to an element in $K$, each
is in some maximal compact subgroup. Therefore, each of them is in
the unique maximal compact $B$. However, since $B\subset G$ is compact
it extends to a maximal compact subgroup of $G$, and since all maximal
compact subgroups are conjugate, there exists a $g\in G$ so that
$gBg^{-1}\subset K$ which in turn implies that $g\rho(\gamma_{i})g^{-1}\in K$
for all $1\leq i\leq r$. By definition, this implies $\rho\in G(\hom(\Gamma,K))$
and as such $G(\hom(\Gamma,K))\supset\hom(\Gamma,G(K))$. \end{proof}
\begin{rem}
\label{realcase:step3} Lemma \ref{ghomkhomgk} remains valid when
replacing the pair $G,K$ by the pair $G(\R),K(\R)$. First note that
we can arrange for $K(\R)\subset K$, and so we assume this is the
case. The key properties that are needed to generalize the proof of
Lemma \ref{ghomkhomgk} are in Section 3.7 of \cite{PeSo}; namely,
(1) $\{\rho(\gamma_{1}),...,\rho(\gamma_{r})\}\subset G(\R)(K(\R))$
generates an Abelian subgroup $A$ consisting of semisimple elements
whose Zariski closure $\overline{A}$ has its $\R$-points in $G(\R)$,
and (2) the intersection of $A$ with the unique maximal compact subgroup
of $\overline{A}$ gives a unique maximal compact subgroup of $A$.
The rest of the argument is the same. 
\end{rem}
Summarizing the last two sections, we have shown that when $\Gamma$
is a finitely generated Abelian group, we can replace the right inclusion,
with the the left inclusion in the following diagram:

\[
\xymatrix{\hom(\Gamma,G_{ss})/G\ar[r]^{\cong} & \hom(\Gamma,G)\quot G\\
\hom(\Gamma,G(K))/G\ar@{^{(}->}[u]\ar[r]_{\cong} & \hom(\Gamma,K)/K\ar@{^{(}->}[u]}
.\]

We will now show that there is a $G$-equivariant strong deformation
retraction $\hom(\Gamma,G(K))\hookrightarrow\hom(\Gamma,G_{ss})$.

\section{Deformation Retraction of Character Varieties\label{sec:Four}}

Recall that a strong deformation retraction (SDR) from a topological
space $M$ to a subspace $N\,\subset\, M$ is a continuous map $\phi\,:\,[0,1]\times M\,\longrightarrow\, M$
such that (1) $\phi_{0}$ is the identity on $M$, (2) $\phi_{t}(n)=n$
for all $n\in N$ and $t\in[0,1]$, and (3) $\phi_{1}(M)\subset N$.
In short, it is a homotopy relative to $N$ between the identity on
$M$ and a retraction mapping to $N$. We are going to construct an
explicit $G$-equivariant strong deformation retraction from $G_{ss}$
to $G(K)$.

Following Pettet-Souto \cite{PeSo}, we start with a deformation in
the case when $G=\SL(n,\mathbb{C})$. Let $\Delta_{n}$ be the subgroup
of diagonal matrices in $\SL(n,\mathbb{C})$, which is a maximal torus,
identified in the usual way with a subgroup of $\left(\mathbb{C}^{*}\right)^{n}$.
Consider the following deformation retraction from $\Delta_{n}$ to
the subset $\Delta_{n}\cap\SU(n)$: \begin{eqnarray}
\sigma:[0,1]\times\Delta_{n} & \to & \Delta_{n}\label{eq:sigma}\\
\left(t,g\right) & \mapsto & \sigma_{t}(g)\nonumber \end{eqnarray}
 where, for $g=\mbox{diag}(z_{1},\cdots,z_{n})\in\Delta_{n}$, and
$z_{1},...,z_{n}\in\mathbb{C}^{*}$, \[
\sigma_{t}(g):=\mbox{diag}(|z_{1}|^{-t}z_{1},\cdots,|z_{n}|^{-t}z_{n}).\]
 The strong deformation retraction properties of $\sigma_{t}$ are
easily established. Note that $\sigma_{t}$ is a homomorphism for
every $t\in[0,1]$, and clearly is equivariant with respect to complex
conjugation in $\SL(n,\C)$.

Suppose that $g\in\SL(n,\mathbb{C})_{ss}$ is semisimple, which means
it is diagonalizable. Since all maximal tori are conjugate, there
is $h\in\SL(n,\mathbb{C})$ (depending on $g$) so that $hgh^{-1}\in\Delta_{n}$.
Define the following map: \begin{equation}
\delta:[0,1]\times\SL(n,\mathbb{C})_{ss}\to\SL(n,\mathbb{C})\label{eq:delta}\end{equation}
 by letting \[
\delta_{t}(g):=h^{-1}\sigma_{t}(hgh^{-1})h.\]
 We have the following properties of $\delta$. 
\begin{lem}
\label{lem:delta-SLn}The map $\delta$ satisfies: 
\begin{enumerate}
\item $\delta_{t}$ is well defined; that is, it does not depend on the
choice of $h$; 
\item $\delta_{t}(g)\in\SL(n,\mathbb{C})_{ss}$ for all $t\in[0,1]$ and
all $g\in\SL(n,\mathbb{C})_{ss}$; 
\item $\delta$ is a strong deformation retraction from $\SL(n,\mathbb{C})_{ss}$
to the set of $\SL(n,\mathbb{C})$ conjugates of $\SU(n)$.

Moreover, for every $t\in[0,1]$, we have:

\item $\delta_{t}$ is $\SL(n,\mathbb{C})$-equivariant; 
\item if $g_{1}g_{2}=g_{2}g_{1}$ then $\delta_{t}(g_{1})\delta_{t}(g_{2})=\delta_{t}(g_{2})\delta_{t}(g_{1})$; 
\item $\delta_{t}(g^{m})=\left(\delta_{t}(g)\right)^{m}$ for all $m\in\mathbb{N}$,
and all $g\in\SL(n,\mathbb{C})_{ss}$ 
\end{enumerate}
\end{lem}
\begin{proof}
All these properties follow from simple computations. For properties
(1) to (5) the reader may consult \cite{PeSo}, or Appendix \ref{appendix1}.
Let us show (6). If $g\in\SL(n,\mathbb{C})$ is semisimple, and $h\in\SL(n,\mathbb{C})$
is chosen so that $hgh^{-1}\in\Delta_{n}$, then $(hgh^{-1})^{m}=hg^{m}h^{-1}$
is also in $\Delta_{n}$. So, \[
\delta_{t}(g^{m})=h^{-1}\sigma_{t}(hg^{m}h^{-1})h=h^{-1}\left(\sigma_{t}(hgh^{-1})\right)^{m}h=\left(\delta_{t}(g)\right)^{m}\]
 as wanted. Note that we used the homomorphism property of $\sigma_{t}$. 
\end{proof}
Let $m\in\mathbb{N}$. For any group $G$, one can define the $m^{{\rm th}}$
power map $p_{m}:G\to G$, by $p_{m}(g)=g^{m}.$ If $g,h\in G$ and
$h^{m}=g$ we say that $h$ is an $m$th root of $g$.

For later convenience, we here record the following fact about the
power map on $\SL(n,\mathbb{C})$. Recall that $\Delta_{n}\subset\SL(n,\mathbb{C})$
denotes the subgroup of diagonal matrices in $\SL(n,\mathbb{C})$. 
\begin{lem}
\label{lem:power-map-diag}Let $p_{m}:\SL(n,\mathbb{C})\to\SL(n,\mathbb{C})$
be the $m^{{\rm th}}$ power map. Then $p_{m}^{-1}(\Delta_{n})=\Delta_{n}$.\end{lem}
\begin{proof}
One inclusion is clear. If $g\in\Delta_{n}$ then $g^{m}\in\Delta_{n}$
so that $g\in p_{m}^{-1}(\Delta_{n})$ by definition. For the converse,
assume that $g\in p_{m}^{-1}(\Delta_{n})$, which means that $g^{m}\in\Delta_{n}$,
so $g^{m}$ is semisimple. Let $g=g_{s}g_{u}$ be the multiplicative
Jordan decomposition, so that $g_{s}$ is semisimple, $g_{u}$ is
unipotent and $g_{s}g_{u}=g_{u}g_{s}$. Then $g^{m}=(g_{s}g_{u})^{m}=g_{s}^{m}g_{u}^{m}$.
Since powers of diagonalizable and unipotent matrices remain respectively
diagonalizable and unipotent, $g_{s}^{m}$ is semisimple and $g_{u}^{m}$
is unipotent. By the uniqueness of the Jordan decomposition, we conclude
that $g^{m}=g_{s}^{m}$ and $g_{u}^{m}=\mathbf{1}$. Since the exponential
map is a diffeomorphism between nilpotent and unipotent matrices in
$\SL(n,\mathbb{C})$, we have $g_{u}^{m}=\exp(m\exp^{-1}(g_{u}))=\mathbf{1}$
which implies that $\exp^{-1}g_{u}=0$ and so $g_{u}=\mathbf{1}$.
Therefore, $g=g_{s}$ is also semisimple. Since $g$ and $g^{m}$
share the same eigenvectors (a linear algebra argument) and $g^{m}\in\Delta_{n}$,
we see that $g$ is also in $\Delta_{n}$. 
\end{proof}

\subsection{Deformation retraction for general $G$}

\label{subsect41}

Now let $G$ be a complex reductive algebraic group, not necessarily
connected, and let $G_{0}$ be its identity component. Let $\mathbf{1}$
denote also the identity in $G$. Then $G_{0}$ is normal in $G$,
and indeed there is a short exact sequence \[
\mathbf{1}\to G_{0}\to G\stackrel{\pi}{\to}F\to\mathbf{1}\]
 where $F:=G/G_{0}\cong\pi_{0}(G)$ is a finite group. As an algebraic
set, $G$ is isomorphic to the Cartesian product $G_{0}\times F$,
and we can write \[
G=\amalg_{f\in F}G_{f}\]
 where $G_{f}$ denotes the connected component $G_{f}:=\pi^{-1}(f)$
for some $f\in F$ (in particular $G_{0}=\pi^{-1}(\mathbf{1}_{F})$).

However, in general $G$ is not the direct product of $F$ and $G_{0}$
as groups. On the other hand, we have a simple relation between semisimple
elements in $G$ and in $G_{0}$, which will enable us to deduce an
appropriate deformation retraction for such a general $G$. Recall
that $G_{ss}$ denotes the set of semisimple elements in $G$. 
\begin{lem}
\label{lem:semisimple-order-m}Let $f\in F$ have order $m$, so that
$f^{m}=\mathbf{1}_{F}$. If $g\in G_{ss}\cap G_{f}$, then $g^{m}\in(G_{0})_{ss}$.
In particular, if $N$ is the order of the group $F$, $G_{ss}^{N}\subset(G_{0})_{ss}$.\end{lem}
\begin{proof}
Let $g\in G_{ss}\cap G_{f}$. Then $\pi(g^{m})=(\pi(g))^{m}=f^{m}=\mathbf{1}$,
so by exactness of the sequence, $g^{m}\in G_{0}$. Since $g$ is
semisimple, it is clear that all its powers are also semisimple. So,
$g^{m}\in(G_{0})_{ss}$. The second statement follows from the fact
that if $N=\#F$, then $f^{N}=\mathbf{1}_{F}$, for all $f\in F$. 
\end{proof}
Consider now an affine algebraic embedding of groups $G\subset\SL(n,\mathbb{C})$.
Let $K$ and $T_{0}$ be respectively, a maximal compact subgroup
of $G$ and a maximal torus of $G_{0}$ such that $T_{0}\cap K$ is
a maximal torus in $K_{0}=K\cap G_{0}$. Up to conjugating in $\SL(n,\mathbb{C})$
we may assume that $T_{0}\subset\Delta_{n}$ and that $K\subset\SU(n)$. 
\begin{lem}
\label{lem:semisimple-inclusion}We have the inclusion $G_{ss}\subset\SL(n,\mathbb{C})_{ss}$.\end{lem}
\begin{proof}
Given $g\in G_{ss}$, suppose that $g\in G_{f}$ with $f^{m}=\mathbf{1}_{F}$.
Then, by Lemma \ref{lem:semisimple-order-m} $g^{m}\in(G_{0})_{ss}$.
So, because $G_{0}$ is connected and all maximal tori are conjugate
inside it, there exists $h\in G_{0}$ so that $hg^{m}h^{-1}\in T_{0}\subset\Delta_{n}$.
This means that $hgh^{-1}\in p_{m}^{-1}(T_{0})\subset p_{m}^{-1}(\Delta_{n})=\Delta_{n}$,
by Lemma \ref{lem:power-map-diag}. So $hgh^{-1}$ is semisimple in
$\SL(n,\mathbb{C})$, and the same holds for $g$. 
\end{proof}
This Lemma allows us to restrict the map $\delta$ in equation (\ref{eq:delta})
to $G_{ss}$, and so we define a new map $\delta:[0,1]\times G_{ss}\to\SL(n,\mathbb{C})_{ss}$,
still denoted by $\delta$. In particular, for every $t\in[0,1]$,
$\delta_{t}$ preserves commutativity (ie, if $g_{1}g_{2}=g_{2}g_{1}$
with $g_{1},g_{2}\in G_{ss}$ then $\delta_{t}(g_{1})\delta_{t}(g_{2})=\delta_{t}(g_{2})\delta_{t}(g_{1})$)
and torsion (ie, $g^{N}=\mathbf{1}$ implies that $\delta_{t}(g^{N})=\left(\delta_{t}(g)\right)^{N}=\mathbf{1}$
for all $m\in\mathbb{N}$, and all $G_{ss}$) automatically, by items
(5) and (6) of Lemma \ref{lem:delta-SLn}. Also, for $t=0$ it is
clear that $\delta_{0}(g)=g$ for all $g\in G_{ss}$, since $\delta_{0}$
is the identity map on $\SL(n,\mathbb{C})$.

Now we complete the proof that $\delta$ indeed gives a strong deformation
retract from $G_{ss}$ to $G(K)$. 
\begin{lem}
\label{lem:Appendix2}We have:

$(1)$ For all $t\in[0,1]$ and all $g\in G_{ss}$, $\delta_{t}(g)\in G_{ss}$;

$(2)$ $\delta_{1}(g)\in G(K)$ for all $g\in G_{ss}$.\end{lem}
\begin{proof}
This was shown in \cite{PeSo}. But see also Appendix \ref{sub:A2},
for a self-contained independent proof. 
\end{proof}
Finally, define $\delta^{r}:[0,1]\times G_{ss}^{r}\to G_{ss}^{r}$
by $\delta_{t}^{r}(g_{1},...,g_{r})=(\delta_{t}(g_{1}),...,\delta_{t}(g_{r}))$
and let $G$ act on $G_{ss}^{r}$ by simultaneous conjugation in each
factor. 
\begin{cor}
\label{defretract} $\delta^{r}:[0,1]\times G_{ss}^{r}\to G_{ss}^{r}$
is a $G$-equivariant deformation retraction onto $G(K)^{r}$ that
preserves torsion and commutativity. \end{cor}
\begin{proof}
The result follows immediately from the the last lemmas. \end{proof}
\begin{rem}
A version of Corollary \ref{defretract} also appears in \cite{PeSo}.
The fact that $\delta^{r}$ preserves torsion is only verified here,
and the fact that it applies to disconnected groups $G$ was independently
discovered by us and the authors of \cite{PeSo} (see also Appendix
\ref{appendix}). 

\label{realcase:step4} The mapping $\delta_{t}$ is also equivariant
with respect to complex conjugation. Therefore it restricts to a strong
deformation retraction on the corresponding fixed locus of complex
conjugation. Therefore Corollary \ref{defretract} extends to the
case of $G(\R)$ and $K(\R)$, since $G(\R)$ is the complex conjugation
fixed point locus of $G$, and likewise $G(\R)(K(\R))$ is the complex
conjugation fixed locus of $G(K)$. 
\end{rem}

\subsection{Proof of the Main Theorem}

\begin{sloppypar} Recall that Theorem \ref{thm:Main} states that
$\X_{\Gamma}(G)$ strongly deformation retracts to $\X_{\Gamma}(K)$
for any complex or real reductive algebraic group $G$, a maximal
compact subgroup $K\subset G$, and any finitely generated Abelian
group $\Gamma$.
\begin{proof}
{[}Proof of Theorem \ref{thm:Main}{]} We handle the complex case
first. Suppose that $\Gamma$ is generated by $r$ elements, as has
been our convention. Proposition \ref{polystablelemma} and Lemma
\ref{semisimplelemma} imply that $\hom(\Gamma,G)\quot G\cong\hom(\Gamma,G)^{ps}/G\cong\hom(\Gamma,G_{ss})/G$,
and Proposition \ref{gkembeds} and Lemma \ref{ghomkhomgk} imply
that $\hom(\Gamma,K)/K\cong\hom(\Gamma,G(K))/G$. In Subsection \ref{subsect41},
we proved that the mapping $\delta_{t}^{r}:\hom(\Gamma,G_{ss})\to\hom(\Gamma,G_{ss})$
is a $G$-equivariant strong deformation retraction onto $\hom(\Gamma,G(K))$.
Therefore, we have a strong deformation retraction from $\hom(\Gamma,G_{ss})/G$
onto $\hom(\Gamma,G(K))/G$. Therefore, $\X_{\Gamma}(G)\cong\hom(\Gamma,G_{ss})/G$
strongly deformation retracts onto $\X_{\Gamma}(K)\cong\hom(\Gamma,G(K))/G$.

The real case follows from Remarks \ref{realcase:step1}, \ref{realcase:step2},
\ref{realcase:step3}, and \ref{realcase:step4}. 
\end{proof}
\end{sloppypar}
\begin{rem}
Theorem \ref{thm:Main} corrects and generalizes the proof of Proposition
$7.1$ in \cite{FlLa}. In that paper, the authors consider free Abelian
groups as an example to situate the main theorems for free groups
(non-Abelian) of that work. The statement and its proof are correct
in some cases, for example when $G$ is the general linear group or
special linear group. After this work was completed and uploaded in
arXiv, M. Bergeron shared with us his article \cite{Ber}, which proves,
by different techniques, a generalization of Theorem \ref{thm:Main}
to nilpotent groups.
\end{rem}
Since a deformation retraction is in particular a homotopy equivalence,
the spaces $\X_{\Gamma}(G)$ and $\X_{\Gamma}(K)$ have the same homotopy
type. So, Theorem \ref{thm:Main} implies they have the same homotopy
groups $\pi_{*}(\X_{\Gamma}(K))\cong\pi_{*}(\X_{\Gamma}(G))$, and
cohomology rings $H^{*}(\X_{\Gamma}(K))\cong H^{*}(\X_{\Gamma}(G))$
for all values of $*$. Moreover, these results do not depend on the
choices made to define $\delta$ since different choices lead to homeomorphic
spaces. In particular, these spaces have the same number of path components.
Recall that we are using the Euclidean topology on both $\X_{\Gamma}(G)$
and $\X_{\Gamma}(K)$. Also, note that the path components in these
spaces are the same as the connected components, since the two notions
are equivalent in the context of semi-algebraic sets (which are always
locally path connected), such as $\X_{\Gamma}(K)$. The following
is immediate. 
\begin{cor}
\label{cor:bijection-comp}There is a natural bijection between the
connected components of $\X_{\Gamma}(G)$ and those of $\X_{\Gamma}(K)$.
Corresponding components have the same homotopy type, and the deformation
retraction preserves each component of $\X_{\Gamma}(G)$. 
\end{cor}
In particular, for every representation, there is another one in the
same component, but taking values in $K$ as we next show.
\begin{cor}
Let $G=K_{\C}$ be a complex reductive algebraic group and let $\Gamma$
be a finitely generated Abelian group. If $C_{\C}\subset\X_{\Gamma}(G)$
denotes a path component, then there exists $\rho\in\hom(\Gamma,K)$
such that $[\![\rho]\!]\in C_{\C}$.\end{cor}
\begin{proof}
Theorem \ref{thm:Main} implies that for every component $C_{\C}\subset\X_{\Gamma}(C)$,
there exists a component $C\subset\X_{\Gamma}(K)$ such that $C\subset C_{\C}$.
The result follows. 
\end{proof}
Let us write $|X|$ for the number of path components of a locally
path-connected space. 
\begin{lem}
Let $G$ be a connected reductive algebraic group and $X$ an affine
$G$-variety. Then $|X|=|X\quot G|$.\end{lem}
\begin{proof}
Since the quotient map $\pi:X\to X\quot G$ is continuous, a given
path component is mapped to a single path component. Being also surjective,
we obtain $|X|\geq|X\quot G|$. Suppose, by contradiction that inequality
holds. Then, at least two path components, say $A,B\subset X$, are
being identified by the action of $G$ or by the further GIT equivalence.
This means that there are $a\in A$, $b\in B$, and a sequence $\{g_{n}\}\subset G$
such that ${\displaystyle \lim_{n\to\infty}g_{n}\cdot a=b}$ (or possibly
with the roles of $a$ and $b$ reversed). Note this includes the
possibility that $b$ is in the $G$-orbit of $a$ by considering
the constant sequence. Returning to the argument, we conclude $b\in\overline{G\cdot a}^{E}$.
Since $G$ is connected, the orbit $G\cdot a=\{g\cdot a:\ g\in G\}$
is connected, and thus $\overline{G\cdot a}^{E}$ is also connected.
Therefore, $\overline{G\cdot a}^{E}\subset\overline{A}^{E}=A$ since
$A$ is closed. This gives a contradiction as $A$ and $B$ are disjoint
in $X$. \end{proof}
\begin{rem}
A similar argument shows that if $X$ is a $G$-space with $G$ connected
then $|X|=|X/G|$. 
\end{rem}
In particular, using Corollary \ref{cor:bijection-comp}, we have
proved the following proposition. 
\begin{prop}
\label{pathcomponents} Let $G$ be a connected complex reductive
algebraic group with maximal compact subgroup $K$, and let $\Gamma$
be a finitely generated Abelian group. Then, $|\hom(\Gamma,G)|=|\X_{\Gamma}(G)|=|\X_{\Gamma}(K)|=|\hom(\Gamma,K)|$. 
\end{prop}

\subsection{\label{sub:RAAG}Right angled Artin groups}

We end this section with a conjecture, based on a remark mentioned
in \cite{PeSo}. Let us define a right-angled Artin group (RAAG) with
torsion to be a finitely generated group which admits a finite presentation
with generators $\gamma_{1},\cdots,\gamma_{r}$ and where every relation
is either $\gamma_{i}\gamma_{j}=\gamma_{j}\gamma_{i}$, or a torsion
relation of the form $\gamma_{j}^{m}=1$ (for some $i,j$ and $m\geq2$).
These groups include free products of cyclic groups and Abelian groups
as extremes. Combining the main result of \cite{FlLa} with Theorem
\ref{thm:Main}, we have shown that there is strong deformation retraction
from $\X_{\Gamma}(G)$ to $\X_{\Gamma}(K)$ when $\Gamma$ is either
a free product of infinite cyclic groups (a free group), or when $\Gamma$
is Abelian. 
\begin{conjecture}
Let $G$ be a complex reductive algebraic group and let $K$ be a
maximal compact subgroup. Let $\Gamma$ be a right-angled Artin group
with torsion. There is strong deformation retraction from $\X_{\Gamma}(G)$
to $\X_{\Gamma}(K)$. 
\end{conjecture}
Providing further evidence for this conjecture, we now use our main
theorem and the main result of \cite{PeSo} to prove two theorems.
First, we need a lemma. Recall that a weak deformation retraction
between a space $X$ and a subspace $A$ is a continuous family of
mappings $F_{t}:X\to X$, $t\in[0,1]$, such that $F_{0}$ is the
identity on $X$, $F_{1}(X)\subset A$, and $F_{t}(A)\subset A$ for
all $t$. 
\begin{lem}
\label{gtogs} Let $\Gamma$ be a finitely generated Abelian group,
and let $G$ be a complex reductive algebraic group with maximal compact
$K$. Then, there exists a $G$-equivariant weak deformation retraction
from $\hom(\Gamma,G)$ onto $\hom(\Gamma,G_{ss})$ that fixes $K$
during the retraction. \end{lem}
\begin{proof}
In \cite{PeSo} it is shown that for every $g\in G$ and every $\epsilon>0$,
there is a continuous $G$-equivariant mapping $g\mapsto(g_{s}^{\epsilon},g_{u}^{\epsilon})$
such that: 
\begin{enumerate}
\item $g=g_{s}^{\epsilon}g_{u}^{\epsilon}=g_{u}^{\epsilon}g_{s}^{\epsilon}$, 
\item if $h$ commutes with $g$ it also commutes with $g_{s}^{\epsilon}$
and $g_{u}^{\epsilon}$, 
\item $g_{s}^{\epsilon}\in G_{ss}$, and 
\item if $g\in H$ where $H$ is an algebraic subgroup, then $g_{s}^{\epsilon}$,
$g_{u}^{\epsilon}\in H$. 
\item the spectral radius of $g_{u}^{\epsilon}-\mathbf{1}$ is at most $\epsilon$. 
\end{enumerate}
This is called the \textit{approximate Jordan decomposition}, since
$g_{u}^{\epsilon}$ is not unipotent, but is approximately so by (5).
Using this, they show that for sufficiently small $\epsilon>0$, $F_{t}(g)=g_{s}^{\epsilon}\mathrm{Exp}((1-t)\mathrm{Log}(g_{u}^{\epsilon}))$
is a $G$-equivariant \textit{weak} deformation retraction from $G$
to $G_{ss}$ that preserves commutativity and point-wise fixes $K$
for all $t$. Note that $\mathrm{Exp}$ and $\mathrm{Log}$ are the
usual power series functions in terms of matrices, which makes sense
in this context since we are choosing an embedding of $G$ as a matrix
group. Returning to the proof, we consequently obtain a $G$-equivariant
weak deformation retraction from $\hom(\mathbb{Z}^{r},G)$ to $\hom(\mathbb{Z}^{r},G_{ss})$
that keeps $\hom(\mathbb{Z}^{r},K)$ point-wise fixed for all $t$.

To prove the lemma, it then suffices to prove that $F_{t}$ preserves
the algebraic torsion relations for all $t$. So suppose $g^{m}=\mathbf{1}$
for some $m$. Then since $H=\langle\mathbf{1}\rangle$ is an algebraic
subgroup, property (4) of approximate Jordan decomposition gives $(g^{m})_{s}^{\epsilon}=\mathbf{1}$
and $(g^{m})_{u}^{\epsilon}=\mathbf{1}$ . Therefore, $F_{t}(g^{m})=\mathbf{1}$.
On the other hand, since $g_{s}^{\epsilon}$ and $g_{u}^{\epsilon}$
commute, we have $\mathbf{1}=g^{m}=(g_{s}^{\epsilon})^{m}(g_{u}^{\epsilon})^{m}$,
and so $(g_{u}^{\epsilon})^{m}=\mathbf{1}$ if $(g_{s}^{\epsilon})^{m}=\mathbf{1}$.
Now considering the subgroup $H=\langle g\rangle=\{\mathbf{1},g,\cdots,g^{m-1}\}$,
which is algebraic since it is finite (it has at most $m$ elements,
as $g^{m}=\mathbf{1}$), property (4) again shows that $g_{s}^{\epsilon}$
is in $H$. Thus, $(g_{s}^{\epsilon})^{m}=\mathbf{1}$ since for all
$h\in H$, $h=g^{k}$ for some $k$, and so $h^{m}=(g^{m})^{k}=\mathbf{1}$.

Moreover, $G$-equivariance and the fact that $g_{s}^{\epsilon}$
and $g_{u}^{\epsilon}$ commute imply that $g_{s}^{\epsilon}$ commutes
with $F_{t}(g)$. Thus, \[
g_{s}^{\epsilon}\mathrm{Exp}((1-t)\mathrm{Log}(g_{u}^{\epsilon}))=\mathrm{Exp}((1-t)\mathrm{Log}(g_{u}^{\epsilon}))g_{s}^{\epsilon}\]
 for all $t$. We conclude: \begin{eqnarray*}
F_{t}(g)^{m} & = & (g_{s}^{\epsilon})^{m}(\mathrm{Exp}((1-t)\mathrm{Log}(g_{u}^{\epsilon})))^{m}\\
 & = & \mathbf{1}^{m}\mathrm{Exp}((1-t)\mathrm{Log}((g_{u}^{\epsilon})^{m}))\\
 & = & \mathrm{Exp}((1-t)\mathrm{Log}(\mathbf{1}))\\
 & = & \mathbf{1}.\end{eqnarray*}
Putting these observations with each other leads to $F_{t}(g^{m})=\mathbf{1}=F_{t}(g)^{m}$,
as desired. \end{proof}
\begin{thm}
\label{raag} Let $\Gamma$ be a RAAG with torsion, $G$ be a complex
reductive algebraic group and let $K$ be a maximal compact subgroup
of $G$ . Then $\X_{\Gamma}(G)$ strongly deformation retracts onto
$\hom(\Gamma,G(K))\quot G$ which fixes the subspace $\hom(\Gamma,K)/K$. 
\end{thm}
\begin{sloppypar} 
\begin{proof}
By the above lemma there is a $G$-equivariant weak deformation retraction
from $G$ to $G_{ss}$ that fixes $K$, preserves torsion, and preserves
commutativity. Putting this together with the $G$-equivariant strong
deformation retraction from $G_{ss}$ to $G(K)$ that also fixes $K$,
preserves torsion, and preserves commutativity from Theorem \ref{thm:Main},
gives a $G$-equivariant weak deformation retraction from $G$ to
$G(K)$ that fixes $K$ and also preserves torsion and commutativity.

Let $\Gamma$ be generated by $r$ elements. The relations are either
torsion relations or commutativity relations. Applying the weak deformation
retraction from $G$ to $G(K)$ factor-wise to $\hom(\Gamma,G)\subset G^{r}$
gives a $G$-equivariant weak deformation retraction onto $\hom(\Gamma,G(K))$
that fixes $\hom(\Gamma,K)$ for all time.

Therefore, we obtain a weak deformation retraction from $\hom(\Gamma,G)/G$
onto $\hom(\Gamma,G(K))/G$ which contains $G(\hom(\Gamma,K))/G\cong\hom(\Gamma,K)/K$
as a fixed subspace.

Note that for each $t$ in $[0,1)$ both the map $F_{t}$ in Lemma
\ref{gtogs}, and the map $\delta_{t}$ in Lemma \ref{lem:Appendix2}
are homeomorphisms, and so since they are $G$-equivariant, they send
closed orbits to closed orbits. When $t=1$, continuity
and Proposition \ref{pro:polystable} imply the limit of closed orbits
in the polystable quotient corresponds to a closed orbit.

Thus, restricting this weak deformation retraction to the subspace
of closed orbits then determines a weak retraction from $\X_{\Gamma}(G)$
onto $\hom(\Gamma,G(K))\quot G$ which contains $\hom(\Gamma,K)/K$
as a fixed subspace.

Since this weak retraction establishes that the inclusion mapping
$\hom(\Gamma,G(K))\quot G\hookrightarrow\X_{\Gamma}(G)$ induces an
isomorphism on homotopy groups, and $\hom(\Gamma,G(K))\quot G$ is
a cellular sub-complex of $\hom(\Gamma,G)\quot G$ given they are
semi-algebraic sets, Whitehead's Theorem (see \cite{Hatcher}) implies
there is a strong deformation retraction from $\hom(\Gamma,G)\quot G$
onto $\hom(\Gamma,G(K))\quot G$ which contains $\hom(\Gamma,K)/K$
as a sub-complex. 
\end{proof}
\end{sloppypar} 
\begin{rem}
Given Lemmata \ref{semisimplelemma} and \ref{ghomkhomgk}, the above
theorem includes our main result as a special case since Abelian groups
are RAAG's with torsion. \end{rem}
\begin{thm}
Let $\Gamma$ be a RAAG with torsion, $G$ be a complex reductive
algebraic group and let $K$ be a maximal compact subgroup of $G$.
If there exists a $K$-equivariant weak retraction from $\hom(\Gamma,G(K))$
to $\hom(\Gamma,K)$, then $\X_{\Gamma}(G)$ strongly deformation
retracts onto $\X_{\Gamma}(K)$. \end{thm}
\begin{proof}
The general Kempf-Ness theory (see \cite{FlLa3}) implies that for
any finitely generated $\Gamma$, $\hom(\Gamma,G)$ $K$-equivariantly
strongly deformation retracts onto a $K$-stable subspace $\mathcal{KN}\subset\hom(\Gamma,G)^{ps}$
such that $\mathcal{KN}/K\cong\X_{\Gamma}(G)$. Thus, $\hom(\Gamma,G)/K$
is weakly homotopic to $\X_{\Gamma}(G)$.

Theorem \ref{raag} implies that there exists a $K$-equivariant weak
deformation retraction from $\hom(\Gamma,G)$ onto $\hom(\Gamma,G(K))$.
However, by hypothesis there exists a $K$-equivariant weak retraction
from $\hom(\Gamma,G(K))$ onto $\hom(\Gamma,K)$. Putting this together
we obtain a weak retraction from $\hom(\Gamma,G)/K$ to $\X_{\Gamma}(K)$;
and so $\hom(\Gamma,G)/K$ is weakly homotopic to $\X_{\Gamma}(K)$.

Thus, $\X_{\Gamma}(G)$ is weakly homotopic to $\X_{\Gamma}(K)$.
The computation of the Kempf-Ness set in \cite{FlLa} implies that
$\X_{\Gamma}(K)$ is contained in $\mathcal{KN}$.

Thus, by Corollary 4.10 in \cite{FlLa3} we conclude the result. 
\end{proof}
\begin{sloppypar} 
\begin{rem}
\label{goodwillie} Pettet and Souto in \cite{PeSo} show there is
a weak retract from $\hom(\mathbb{Z}^{r},G(K))$ to $\hom(\mathbb{Z}^{r},K)$,
but they do not show it preserves torsion or that it is $K$-equivariant.
This shows that their result alone does not imply ours. Given the
above theorem, it is natural to ask if one can make their weak retraction
$K$-equivariant and obtain a special case of our main theorem, but
as the following example shows, the existence of a SDR from a $G$-space
$X$ to a $G$-subspace $Y$, and even a SDR on quotients $X/G$ to
$Y/G$, together are not enough to imply the existence of a $G$-equivariant
SDR.

Here is the example: Take a CW space $Z$ that is acyclic but not
contractible. Let $X$ be the suspension of $Z$, and let $G$ be
the order 2 group that acts on the suspension by switching the two
cones. Notice that $Z$ is the fixed point set. Then $X$ is contractible
since the suspension of an acyclic space is contractible (by Whitehead's
Theorem), and the orbit space is contractible since it is a cone.
But $X$ is not $G$-equivariantly contractible because the fixed
point set is not contractible. 
\end{rem}
\end{sloppypar}

\section{Applications to Irreducibility and Topology of $\X_{\Gamma}(G)$.}

We now investigate other interesting consequences of Theorem \ref{thm:Main},
related to the topology and irreducibility of the character varieties
$\X_{\Gamma}(G)$.

We first consider the free Abelian group of rank $r$, $\Gamma=\Z^{r}$.
For these groups, and for a connected and semisimple compact Lie group
$K$, the cases when $\hom(\Z^{r},K)$ is connected have been determined
(see \cite{KaSm}).

This leads us to consider also connected \textit{semisimple} algebraic
groups $G$; by definition, these do not contain any non-trivial closed
connected Abelian normal subgroups, and are clearly reductive. As
usual, we assume that the trivial group is not semisimple. Consequently,
every connected semisimple group has rank at least 1. 
\begin{thm}
\label{connected:cor} Let $G$ be a semisimple connected algebraic
group over $\C$. Then $\X_{\Z^{r}}(G)$ is connected if and only
if $r=1$, $r=2$ and $G$ is simply connected, or $r\geq3$ and $G$
is a product of $\SL(n,\C)$'s and $\mathsf{Sp}(n,\C)$'s. \end{thm}
\begin{proof}
Let $K$ be a maximal compact subgroup of $G$. Then $K$ is connected
and semisimple. According to \cite{KaSm}, $\hom(\Z^{r},K)$ is connected
if and only if $r=1$, $r=2$ and $G$ is simply connected, or $r\geq3$
and $G$ is a product of $\SU(n)$'s and $\mathsf{Sp}(n)$'s. By Proposition
\ref{pathcomponents}, $\hom(\Z^{r},K)$ is connected if and only
if $\X_{\Z^{r}}(K)$ is connected. However, by Corollary \ref{cor:bijection-comp},
$\X_{\Z^{r}}(K)$ is connected if and only if $\X_{\Z^{r}}(G)$ is
connected, and the result follows. 
\end{proof}

\subsection{The identity component}

There is an extremely useful characterization of $\X_{\Z^{r}}(G)$,
recently obtained by A. Sikora \cite{Si5}, analogous to work on the
compact case by T. Baird \cite{Ba1} (see also \cite{AG}).

To describe it, denote by $T$ the maximal torus of $G$, by $N_{G}(T)$
its normalizer in $G$, and by $W=N_{G}(T)/T$ its Weyl group. For
any $r\in\mathbb{N}$, the Weyl group $W$ acts on $T^{r}$ by simultaneous
conjugation. Consider the natural morphisms of affine varieties \begin{eqnarray}
\hom(\mathbb{Z}^{r},T) & \hookrightarrow & \hom(\mathbb{Z}^{r},G)\nonumber \\
\parallel &  & \downarrow\pi\label{eq:normalization}\\
T^{r} & \stackrel{\varphi}{\to} & \X_{\mathbb{Z}^{r}}(G).\nonumber \end{eqnarray}
 Here, the top row is the natural inclusion, $\pi$ the canonical
projection, and $\varphi$ makes the diagram commute. The morphism
$\varphi$ factors through the action of the Weyl group, and Sikora
(see \cite{Si5}) showed the following theorem. Denote by $\X_{\mathbb{Z}^{r}}^{0}(G)\subset\X_{\mathbb{Z}^{r}}(G)$
the image of $\varphi$. 
\begin{thm}
$[$See \cite{Si5}$]$ \label{pro:Sikora} Let $G$ be a complex
reductive algebraic group. Then $\X_{\mathbb{Z}^{r}}^{0}(G)$ is an
irreducible component of $\X_{\mathbb{Z}^{r}}(G)$ and there is a
bijective birational normalization morphism \[
\chi:T^{r}\quot W\to\X_{\mathbb{Z}^{r}}^{0}(G).\]
 Moreover, $\X_{\mathbb{Z}^{r}}(\SL(n,\C))$ and $\X_{\mathbb{Z}^{r}}(\mathsf{Sp}(n,\C))$
are irreducible normal varieties for any $r,n\geq1$. \end{thm}
\begin{rem}
\label{compactcase} Similar reasoning shows that $\hom^{0}(\mathbb{Z}^{r},K)/K$,
the identity component of the compact character variety, is in bijection
with $(T\cap K)^{r}/W$. This statement appears already in \cite[Remarks 3 \& 4]{Ba1}.
Moreover, this implies $\hom^{0}(\mathbb{Z}^{r},K)/K$ is homeomorphic
to $(T\cap K)^{r}/W$ since they are both compact and Hausdorff. 
\end{rem}
The character variety $\X_{\Gamma}(G)$ has a natural base point,
which is the trivial representation defined by $\rho_{e}(\gamma):=e$,
the identity in $G$, for all $\gamma\in\Gamma$ (more precisely,
the base point is the class $[\![\rho_{e}]\!]$ of $\rho_{e}$). Let
us denote the component of $\X_{\Gamma}(G)$ containing this base
point by $\X_{\Gamma}^{e}(G)$, and the corresponding component of
$\hom(\Gamma,G)$ by $\hom^{e}(\Gamma,G)$. Both of these components
are referred to as the \emph{identity component}. 
\begin{defn}
Denote by $W_{r}(G)\subset\hom(\mathbb{Z}^{r},G)\subset G^{r}$ the
subset of commuting $r$-tuples $\rho=(g_{1},...,g_{r})\in G^{r}$
such that there is a maximal torus $T\subset G$ with $g_{1},...,g_{r}\in T$. \end{defn}
\begin{rem}
\label{goldman} One may ask whether one can simultaneously conjugate
commuting elements in a reductive group to a single maximal torus.
This works for $\SL(n,\mathbb{C})$, but fails in general. For instance
consider $\mathrm{diag}(-1,-1,1)$ and $\mathrm{diag}(1,-1,-1)$ in
$\mathsf{SO}(3)$. They commute, and the first is in $\mathsf{SO}(2)\times\{1\}$
and the second is in $\{1\}\times\mathsf{SO}(2)$, but the two elements
cannot be simultaneously conjugated by $\mathsf{SO}(3)$ into the
same maximal torus of $\mathsf{SO}(3)$. 
\end{rem}
Recall the definition of $\X_{\mathbb{Z}^{r}}^{0}(G)$ before Theorem
\ref{pro:Sikora}. From Equation (\ref{eq:normalization}), we deduce
that $W_{r}(G)=\pi^{-1}(\X_{\mathbb{Z}^{r}}^{0}(G))$ where $\pi:\hom(\mathbb{Z}^{r},G)\to\X_{\Gamma}(G)$
is the quotient morphism. Therefore, $W_{r}(G)\quot G=\X_{\mathbb{Z}^{r}}^{0}(G)$.

Since $\X_{\Z^{r}}^{0}(G)$ is irreducible and contains the identity
it is connected, and thus by definition $\X_{\Z^{r}}^{0}(G)\subset\X_{\Z^{r}}^{e}(G).$ 
\begin{lem}
\label{lem:B}$W_{r}(G)\subset\hom^{e}(\mathbb{Z}^{r},G)$. \end{lem}
\begin{proof}
If $\rho\in W_{r}(G)$, let $T$ be such that $\rho=(g_{1},...,g_{r})\in T^{r}$.
Since $T$ is connected, just let $\gamma_{1},...,\gamma_{r}$ be
paths inside $T$ joining $g_{1}$ to $e\in G,...,g_{r}$ to $e$.
These paths will form, component-wise, a path of commuting $r$-tuples
inside $T^{r}$ joining $\rho$ to $(e,...,e)$. So, $\rho$ is in
$\hom^{e}(\mathbb{Z}^{r},G)$. 
\end{proof}
Notice that $\X_{\Z^{r}}^{e}(G)=\X_{\Z^{r}}^{0}(G)$ if we also had
$\hom^{e}(\mathbb{Z}^{r},G)\subset W_{r}(G)$, by Lemma \ref{lem:B}.
This leads naturally to the following problem. 
\begin{problem}
Determine which $r$ and $G$ give the equality $\X_{\Z^{r}}^{e}(G)=\X_{\Z^{r}}^{0}(G)$. 
\end{problem}
In \cite{Si5} it is shown that this equality holds when $r=1,2$
and $G$ is connected reductive (see also \cite{Ri79}), or for any
$r$ when $G$ equals $\SL(n,\C)$, $\GL(n,\C)$ or $\mathsf{Sp}(n,\C)$.
More generally, we prove below that this equality holds, for any $r$,
when the derived group of $G$ is a product of $\SL(n,\C)$'s and
$\mathsf{Sp}(n,\C)$'s (see Corollary \ref{cor:reductive}).

\subsection{Irreducibility of $\X_{\Gamma}(G)$ for semisimple connected $G$}

Theorem \ref{connected:cor} lists the cases where $\X_{\Gamma}(G)$
is connected with $G$ semisimple connected and $\Gamma$ free Abelian.
To generalize the result to the case of a general Abelian $\Gamma$,
we need the following lemma. 
\begin{lem}
\label{lem:not-free}If $\Gamma$ is a finitely generated Abelian
group which is not free $($i.e., it has some torsion$)$, then $\X_{\Gamma}(G)$
is not path-connected. 
\end{lem}
\begin{sloppypar} 
\begin{proof}
Since $G$ is connected the conclusion follows from Proposition \ref{pathcomponents}
if $\hom(\Gamma,G)$ is disconnected. Say $\Gamma$ has non-trivial
torsion. Then, it can be written in the form $\Gamma=\mathbb{Z}_{d}\oplus\Gamma'$,
for some $d\in\mathbb{N}\setminus\{0,1\}$, and denote by $\delta\in\Gamma$
a generator of the $\mathbb{Z}_{d}$-summand, so that $\delta^{d}=1$.
Denote by $\rho_{0}\in\hom(\Gamma,G)$ the trivial representation
$\rho_{0}(\gamma)=e\in G$ for all $\gamma\in\Gamma$. In order to
define a general representation $\rho$ of $\Gamma$ into $G$, we
just need to assign the value at $\delta$, say $\rho(\delta)=g\in G$,
satisfying $g^{d}=e$, together with $\rho'\in\hom(\Gamma',G)$, as
long as $g$ commutes with every $\rho(\gamma')$ for $\gamma'\in\Gamma'$.

Assume now that $G$ is faithfully embedded as an algebraic representation
in $\GL(n,\mathbb{C})$, in such a way that a maximal torus $T\subset G$
is a subtorus of $\Delta_{n}$, the torus of diagonal elements in
$\GL(n,\mathbb{C})$.

We can always arrange for this since a maximal torus is a connected
maximal Abelian subgroup and thus contained in a connected maximal
Abelian subgroup of $\GL(n,\mathbb{C})$, a maximal torus, and all
maximal tori in $\GL(n,\mathbb{C})$ are conjugate. So let \begin{eqnarray*}
\varphi:T & \to & \Delta_{n}\\
h & \mapsto & \left(\varphi_{1}(h),\cdots,\varphi_{n}(h)\right)\end{eqnarray*}
 be the corresponding group homomorphism. Since $\varphi$ is an embedding,
at least one of the components is nontrivial; without loss of generality
let it be $\varphi_{1}:T\to\mathbb{C}^{*}$. Thus $\varphi_{1}$ is
an algebraic character of $T$ and thus has the form $t_{1}^{m_{1}}\cdots t_{n}^{m_{n}}$.
This means, since the character is non-trivial, that $\varphi_{1}$
is surjective, and so there is $h_{1}\in T$ such that $\varphi_{1}(h_{1})=e^{\frac{2\pi i}{d}}$.
Now, the assignment $\rho_{1}(\delta)=h_{1}$ and $\rho_{1}(\gamma)=e$
for all $\gamma\in\Gamma'$ defines a representation $\rho_{1}\in\hom(\Gamma,G)$,
since $\rho_{1}(\delta^{d})=h_{1}^{d}=1\in\mathbb{C}^{*}$ (and $e$
commutes with everything).

Suppose that there was a continuous path $\rho_{t}$, $t\in[0,1]$
from $\rho_{0}$ to $\rho_{1}$. Then, as a function of $t$, $(\varphi_{1}\circ\rho_{t})(\delta):[0,1]\to\mathbb{C}^{*}$
is continuous, and has values on $\mathbb{Z}_{d}\subset\mathbb{C}^{*}$the
set of $d$-th roots of unity, since $(\varphi_{1}\circ\rho_{t})(\delta^{d})=\left((\varphi_{1}\circ\rho_{t})(\delta)\right)^{d}=1$.
Since a continuous map sends the connected interval $[0,1]$ to a
connected set, but the image of $(\varphi_{1}\circ\rho_{t})(\delta)$
is disconnected (as $(\varphi_{1}\circ\rho_{0})(\delta)=1$ and $(\varphi_{1}\circ\rho_{1})(\delta)=e^{\frac{2\pi i}{d}}\neq1$),
this contradiction shows that $\hom(\Gamma,G)$ is not connected. 
\end{proof}
\end{sloppypar}

In the special case of Abelian reductive groups, we are able to compute
exactly the number of path components. For related results on counting
components, see Corollary 3.4 in \cite{AG}. 
\begin{lem}
\label{lem:T-char}Let $T$ be a complex reductive Abelian connected
group of dimension $m$ $($or its maximal compact subgroup$)$, and
let $\Gamma=\Z^{s}\oplus\bigoplus_{j=1}^{t}\Z_{n_{j}}$. Then $|\hom(\Gamma,T)|=\prod_{j=1}^{t}n_{j}^{m}$.\end{lem}
\begin{proof}
Given two finitely generated Abelian groups $A,B$, and an Abelian
group $T$, we have $\hom(A\oplus B,T)\cong\hom(A,T)\times\hom(B,T)$.
Given two Lie groups $T_{1},T_{2}$ we also have $\hom(\Gamma,T_{1}\times T_{2})\cong\hom(\Gamma,T_{1})\times\hom(\Gamma,T_{2})$.
If $T$ is a complex Abelian reductive connected group of dimension
$m$, then $T$ is isomorphic to a complex torus, so $T\cong(\mathbb{C}^{*})^{m}$
with maximal compact subgroup $(S^{1})^{m}$. Either way, $\hom(\mathbb{Z},T)\cong T$.
We obtain: \[
\hom(\Gamma,T)\cong T^{s}\times\hom(\oplus_{j=1}^{t}\Z_{n_{j}},\left(\mathbb{C}^{*}\right)^{m})\cong T^{s}\times_{j=1}^{t}\left(\hom(\mathbb{Z}_{n_{j}},\mathbb{C}^{*})\right)^{m}.\]
 Finally, because $T$ is connected and $|\hom(\mathbb{Z}_{n},\mathbb{C}^{*})|=n$
(for any $n\in\mathbb{N}$, the set of $n$-th roots of unity has
cardinality $n$), we obtain the desired formula. 
\end{proof}
Now we are ready to extend Theorem \ref{connected:cor} as follows. 
\begin{thm}
\label{thm:General-Gamma}Let $G$ be a semisimple connected algebraic
group over $\C$ and $\Gamma$ a finitely generated Abelian group
of rank $r$. Then $\X_{\Gamma}(G)$ is path connected if and only
if:

$(1)$ $\Gamma$ is free, and

$(2)$ $r=1$, $r=2$ and $G$ is simply connected, or $r\geq3$ and
$G$ is a product of $\SL(n,\C)$'s and $\mathsf{Sp}(n,\C)$'s. \end{thm}
\begin{proof}
Corollary \ref{connected:cor} shows the statement for $\Gamma=\mathbb{Z}^{r}$.
Now, suppose that $\Gamma$ is not free. Then $\X_{\Gamma}(G)$ cannot
be connected, by the preceding Lemma \ref{lem:not-free}, so the result
follows. 
\end{proof}
Finally we can show Theorem \ref{classification} which says when
$G$ is semisimple and connected and $\Gamma$ is finitely generated
and Abelian, then $\X_{\Gamma}(G)$ is an irreducible variety if and
only if $\Gamma$ is free and either $r=1$, $r=2$ and $G$ is simply
connected, or $r\geq3$ and $G$ is a product of $\SL(n,\C)$'s and
$\mathsf{Sp}(n,\C)$'s. 
\begin{proof}
{[}Proof of Theorem \ref{classification}{]} If a variety is irreducible,
it is path-connected \cite{Sh1,Sh2}. So, by Theorem \ref{thm:General-Gamma},
the given conditions are necessary.

Conversely, assume that $\Gamma$ is free, and either $r=1$, $r=2$
and $G$ is simply connected, or $r\geq3$ and $G$ is a product of
$\SL(n,\C)$'s and $\mathsf{Sp}(n,\C)$'s. In \cite{St}, the case
$r=1$ is shown to be irreducible; in \cite{Ri79} the $r=2$ and
$G$ simply-connected case is shown to be irreducible. Finally, in
the very recent pre-print \cite{Si5} the cases $r\geq3$ and $G=\SL(n,\C)$
or $G=\mathsf{Sp}(n,\C)$ are shown to be irreducible. However, by
Proposition \ref{gproduct}, if $G$ is a product of $\SL(n,\C)$'s
and $\mathsf{Sp}(n,\C)$'s then $\X_{\Z^{r}}(G)$ decomposes into
a product of $\X_{\Z^{r}}(\SL(n,\C))$'s and $\X_{\Z^{r}}(\mathsf{Sp}(n,\C))$'s.
Since the product of irreducible varieties is irreducible, we conclude
that the last case is irreducible too. \end{proof}
\begin{rem}
As is shown in \cite{CoMa}, a complex affine algebraic set $V$ is
an irreducible algebraic variety if and only if $V$ contains a connected
dense open subset which is smooth. Thus, in the context of this last
theorem, we conclude that the smooth points of $\X_{\Z^{r}}(G)$ are
path-connected without exhibiting a smooth path between smooth points. 
\end{rem}
We now can settle a question implicit in \cite{Si5}. 
\begin{cor}
Let $G$ be a complex exceptional Lie group. If $r\geq3$, then $\X_{\mathbb{Z}^{r}}(G)$
is not irreducible, nor path-connected. If $r=2$, then $\X_{\mathbb{Z}^{r}}(G)$
is irreducible $($respectively, path-connected$)$ if and only if
$G$ is $\mathsf{E}_{8},\mathsf{F}_{4}$, or $\mathsf{G}_{2}$.\end{cor}
\begin{proof}
The complex exceptional Lie groups are simple, and hence semisimple
and connected. Moreover, being simple and exceptional implies they
are not products of $\SL(n,\C)$'s and $\mathsf{Sp}(n,\C)$'s. It
is also known that $\mathsf{E}_{8},\mathsf{F}_{4}$, and $\mathsf{G}_{2}$
are simply-connected, while $\mathsf{E}_{6}$ and $\mathsf{E}_{7}$
are not. Therefore, Theorems \ref{thm:General-Gamma} and \ref{classification}
give the result. 
\end{proof}

\subsection{Irreducibility of $\X_{\mathbb{Z}^{r}}(G)$ for connected reductive
$G$}

Suppose now that $G$ is a \emph{connected} reductive group. One can
represent any element $g\in G$ as $g=th$ where $t$ belongs to $Z^{0}$,
the connected component of the center of $G$, and $h$ is in $DG$,
the derived group of $G$ (see \cite[page 200]{Mil}). By definition
$DG=[G,G]$ is the group of commutators of elements in $G$, and it
is well known that $DG$ is a connected semisimple algebraic group.

Moreover, the canonical morphism $\tilde{G}:=Z^{0}\times DG\to G,$
defined by the product is what is called a \emph{central isogeny}
(see \cite[Cor 5.3.3]{Con}). This means that the kernel of the morphism
above is a finite subgroup $F$ of the center of $G$. Recall also
that, because $G$ is reductive, the radical of $G$ coincides with
$Z^{0}$ and this is a torus, hence isomorphic to some $(\mathbb{C}^{*})^{d}$,
$d\geq0$. So, we have an exact sequence\[
\mathbf{1}\to F\to\tilde{G}=(\mathbb{C}^{*})^{d}\times DG\to G\to\mathbf{1}.\]

Suppose now $\Gamma=\langle\gamma_{1},...,\gamma_{r}\rangle\cong\mathbb{Z}^{r}$,
a free Abelian group, and let $\rho\in\hom(\Gamma,G)$. Using the
sequence above, for every $a_{i}=\rho(\gamma_{i})$ there exists $\tilde{a}_{i}=(t_{i},h_{i})\in\tilde{G}=Z^{0}\times DG$
such that $a_{i}=t_{i}h_{i}$. Since $\{t_{1},...,t_{r}\}$ is in
the center of $G$, and since $\{a_{1},...,a_{r}\}$ all commute with
each other, we conclude that for all $1\leq i,j\leq r$ $t_{i}t_{j}h_{i}h_{j}=a_{i}a_{j}=a_{j}a_{i}=t_{i}t_{j}h_{j}h_{i}$,
and thus $h_{i}h_{j}=h_{j}h_{i}$. Therefore, we obtain $\tilde{\rho}=(\tilde{a}_{1},...,\tilde{a}_{r})\in\hom(\Gamma,\tilde{G})$.
Now because $F$ is Abelian (a finite subgroup of $Z$), $\hom(\Gamma,F)=F^{r}$
is indeed an Abelian group itself.

Since $F$ is central, there is a natural action of $\hom(\Gamma,F)$
on $\hom(\Gamma,\tilde{G})$ and indeed, one can easily check that\begin{align}
\hom(\Gamma,G) & =\hom(\Gamma,\tilde{G})/\hom(\Gamma,F)\nonumber \\
 & \cong\left(\hom(\Gamma,(\mathbb{C}^{*})^{d})\times\hom(\Gamma,DG)\right)/\hom(\Gamma,F)\label{eq:reductive}\end{align}
 as affine algebraic varieties.

Since $\hom(\mathbb{Z}^{r},(\mathbb{C}^{*})^{d})=(\mathbb{C}^{*})^{dr}$
is irreducible, the Cartesian product of irreducible varieties is
irreducible, and the quotient of an irreducible variety by reductive
group (this includes the finite group case) is irreducible, formula
(\ref{eq:reductive}) immediately implies the following. 
\begin{lem}
\label{lem:reductive}Let $G$ be a connected reductive group. If
$\hom(\mathbb{Z}^{r},DG)$ is irreducible, then the same holds for
$\hom(\mathbb{Z}^{r},G)$ and $\X_{\mathbb{Z}^{r}}(G)$. 
\end{lem}
Now, we present a partial generalization our main theorem to the case
when $G$ is connected and reductive. 
\begin{cor}
\label{cor:reductive}Let $G$ be a connected reductive group. Suppose
either $r=1$, $r=2$ and $DG$ is simply connected, or $r\geq3$
and $DG$ is a product of $\SL(n,\C)$'s and $\mathsf{Sp}(n,\C)$'s.
Then $\X_{\mathbb{Z}^{r}}(G)$ is irreducible.\end{cor}
\begin{proof}
As above, writing $G$ as $((\mathbb{C}^{*})^{d}\times DG)/F$, we
have:\begin{align*}
\X_{\mathbb{Z}^{r}}(G) & =\hom(\mathbb{Z}^{r},G)\quot G\\
 & =\left[\left(\hom(\mathbb{Z}^{r},(\mathbb{C}^{*})^{d})\times\hom(\mathbb{Z}^{r},DG)\right)/\hom(\mathbb{Z}^{r},F)\right]\quot G\\
 & \cong\left[(\mathbb{C}^{*})^{dr}\times\left(\hom(\mathbb{Z}^{r},DG)\quot G\right)\right]/\hom(\mathbb{Z}^{r},F)\\
 & \cong\left[(\mathbb{C}^{*})^{dr}\times\X_{\mathbb{Z}^{r}}(DG)\right]/\hom(\mathbb{Z}^{r},F)\end{align*}
 In the equation above, the fact that the action of the finite central
group $\hom(\mathbb{Z}^{r},F)$ commutes with the conjugation action
by $G$ justifies both the interchange of the two quotients, and the
fact that $G$ only acts non-trivially on the factor $\hom(\mathbb{Z}^{r},DG)$.
Note also that $G=DG\cdot Z^{0}$ acts on $DG$ by conjugation in
the natural way: writing $g=g_{D}\cdot g_{0}$ with $g_{D}\in DG$
and $g_{0}\in Z^{0}$, we have $g\cdot h:=ghg^{-1}=g_{D}h(g_{D})^{-1}$.
This justifies the identification of quotient spaces $\hom(\mathbb{Z}^{r},DG)\quot G=\X_{\mathbb{Z}^{r}}(DG)$.

Now by Theorem \ref{classification}, in either of the three assumed
situations, $\X_{\mathbb{Z}^{r}}(DG)$ is irreducible. So, both the
Cartesian product $(\mathbb{C}^{*})^{dr}\times\X_{\mathbb{Z}^{r}}(DG)$
and $\X_{\mathbb{Z}^{r}}(G)$ (its quotient by $\hom(\mathbb{Z}^{r},F)$)
are irreducible.\end{proof}

\begin{rem}
\label{rem:no-converse}Note that Corollary \ref{cor:reductive} only
applies to free Abelian groups $\Gamma$, and moreover, it is not
clear whether the converse direction holds. Namely, if $\X_{\mathbb{Z}^{r}}(G)$
is irreducible, it does not seem clear whether it is necessary that the derived group
$DG$ is a product of special linear and symplectic groups (see Remark \ref{rem:converse-for-GL}).
\end{rem}

\subsection{Further applications to algebra and topology}

Let $C_{r,n}$ denote the algebraic subvariety of $\mathbb{C}^{rn^{2}}$
consisting of commuting $r$-tuples of $n\times n$ complex matrices.
It is known that $C_{r,n}$ is irreducible when $r\leq2$ (for any
$n$) and that it is reducible (not irreducible) when $r,n\geq4$
(see \cite{Ger,Gur}). The irreducibility for some cases when $r=3$
is still an open question, as far as we know. Under the inclusions
$\SL(n,\mathbb{C})\subset\GL(n,\mathbb{C})\subset M_{n\times n}(\mathbb{C})=C_{1,n}$,
it is clear that we have natural inclusions\[
\hom(\mathbb{Z}^{r},\SL(n,\mathbb{C}))\subset\hom(\mathbb{Z}^{r},\GL(n,\mathbb{C}))\subset C_{r,n}.\]

\begin{prop}
The variety $C_{r,n}$ is irreducible if and only if $\hom(\mathbb{Z}^{r},\GL(n,\mathbb{C}))$
is dense in $C_{r,n}$, and $\hom(\mathbb{Z}^{r},\SL(n,\mathbb{C}))$ is irreducible.\end{prop}
\begin{proof}
Denote an element of $C_{r,n}$ by $A=(A_{1},\cdots,A_{r})$ so that
$A_{1},\cdots,A_{r}$ are $n\times n$ matrices satisfying $A_{i}A_{j}=A_{j}A_{i}$
for all $i,j=1,...,r$. The following simple construction relates
$C_{r,n}$ with $\hom(\mathbb{Z}^{r},\SL(n,\C))$. Consider the map
\begin{eqnarray*}
\psi:C_{r,n} & \to & \mathbb{C}^{r}\\
(A_{1},...,A_{r}) & \mapsto & (\det A_{1},\cdots,\det A_{r}).\end{eqnarray*}
By evaluating a homomorphism $\rho\in\hom(\mathbb{Z}^{r},\SL(n,\C))$
on a set of commuting generators of $\mathbb{Z}^{r}$, it is clear
that we have an isomorphism of varieties \[
\hom(\mathbb{Z}^{r},\SL(n,\C))\cong\psi^{-1}(1,\cdots,1).\]
Note that $C_{r,n}$ is invariant by dilation, that is, if $(A_{1},\cdots,A_{r})\in C_{r,n}$,
then $(\lambda A_{1},\cdots,\lambda A_{r})\in C_{r,n}$, for any scalar
$\lambda\in\mathbb{C}$. It follows that $\psi$ is surjective $\psi(C_{r,n})=\mathbb{C}^{r}$.
Now suppose that $C_{r,n}$ is irreducible. Then, clearly $\hom(\mathbb{Z}^{r},\GL(n,\C))$
is dense in $C_{r,n}$, being the complement of the trivial determinant
subvariety. Moreover, by Bertini's theorem (\cite{Sh1}, vol I, p.
139) there is a Zariski open subset $U\subset\mathbb{C}^{r}$ such
that $\psi^{-1}(y)$ is irreducible for $y\in U$. It is not obvious
a priori that $(1,\cdots,1)\in U$, but we can observe the fact that
all fibers of $\psi$ of the form $\psi^{-1}(z_{1},\cdots,z_{r})$,
with $z_{1}\cdots z_{r}\neq0$, are indeed isomorphic as algebraic
varieties. This, together with the fact that $U$ is dense, shows
that $\hom(\mathbb{Z}^{r},\SL(n,\C))\cong\psi^{-1}(1,\cdots,1)$ is
irreducible. Conversely, suppose that $\hom(\mathbb{Z}^{r},\SL(n,\C))$
is irreducible. Then, the representation variety $\hom(\mathbb{Z}^{r},\GL(n,\C))$
is also irreducible. Indeed, $\SL(n,\C)$ is the derived group of
$\GL(n,\C)$, so this follows from Lemma \ref{lem:reductive}. Assuming
that $\hom(\mathbb{Z}^{r},\GL(n,\C))$ is a dense subset in
$C_{r,n}$, and using the fact that the Zariski closure of an irreducible
algebraic set is irreducible, we conclude that $C_{r,n}$ is irreducible.\end{proof}

\begin{rem}
Note that the argument above shows that the representation varieties $\hom(\mathbb{Z}^{r},\SL(n,\C))$
and $\hom(\mathbb{Z}^{r},\GL(n,\C))$ are either simultaneously irreducible
or reducible (a slight improvement of Lemma \ref{lem:reductive}, in the case $G=\GL(n,\C)$).
\end{rem}

\begin{rem}
\label{rem:converse-for-GL}
In general, if $V$ is an irreducible
$G$-variety, then $V\quot G$ is irreducible. However, the converse
is not generally true. For example, let $G=\C^*$ act on the reducible
variety $V=\{(x,y)\in\C^{2}\ |\ xy=0\}$ by $t\cdot(x,y)=(x,t\, y)$,
for $x,y\in\C$ and $t\in\C^{*}$. Then, $V\quot G$ is irreducible,
since the punctured axis $\{(0,y)\ |\ y\neq0\}$ is a single non-closed
orbit that gets identified with the origin, the unique closed orbit
in its closure.  Therefore, since $C_{r,n}$ is reducible for $r,n\geq 4$, 
we know either $\hom(\Z^r,\SL(n,\C))$ is reducible or $\hom(\Z^r,\GL(n,\C))\subset C_{r,n}$ is not dense.  
However, it is not immediately clear which holds, despite the fact that $\hom(\Z^r,\SL(n,\C)\quot \SL(n,\C)$ is irreducible itself.
\end{rem}

Another consequence of our main Theorem concerns the fundamental group
of $\X_{\Z^{r}}(G)$.
\begin{cor}
Let $G$ be a connected reductive complex algebraic group, and let
$\X_{\Z^{r}}^{e}(G)\subset\X_{\Z^{r}}(G)$ be the identity component.
Then,

$(1)$ If $G$ is simply connected, then $\X_{\Z^{r}}^{e}(G)$ is
simply connected;

$(2)$ If $G$ is a product of $\SL(n,\C)$'s and $\mathsf{Sp}(n,\C)$'s,
then $\X_{\Z^{r}}(G)$ is simply connected.\end{cor}
\begin{proof}
(1) Let $K$ be a maximal compact subgroup of $G$. In \cite{GoPeSo},
it is shown that the fundamental group of the identity component of
$\hom(\Z^{k},K)$ is $\pi_{1}(K)^{r}$. Thus if $G$ is simply-connected,
then so is $K$ and consequently the identity component of $\hom(\Z^{k},K)$
is simply-connected. In \cite{Br} it is shown that if a connected
compact group acts on a simply connected space $X$, then $X/K$ is
also simply-connected. We therefore conclude that $\X_{\Z^{r}}^{e}(K)$
is simply-connected. Theorem \ref{thm:Main} then implies that $\X_{\Z^{r}}^{e}(G)$
is simply-connected as well.

(2) Suppose $G$ is a a product of $\SL(n,\C)$'s and $\mathsf{Sp}(n,\C)$'s.
From Theorem \ref{classification} $\X_{\Z^{r}}^{e}(G)=\X_{\Z^{r}}(G)$,
so the result follows from (1) and the fact that both $\SL(n,\C)$
and $\mathsf{Sp}(n,\C)$ are simply connected. 
\end{proof}
We finish with a result about the rational cohomology of $\X_{\Z^{r}}(G)$.
For that we introduce some notation. Consider $n,r\geq1$ and $rn$
variables $\alpha_{j}^{i}$, $1\leq i\leq r$ and $1\leq j\leq n$,
and let $\Lambda_{\Q}[\alpha_{1}^{1},...,\alpha_{n}^{r}]$ be the
exterior algebra over $\Q$ on these variables. Let the symmetric
group on $n$ letters, $S_{n}$, act on this algebra by simultaneously
permuting the lower indices in the symbols $\alpha_{j}^{i}$ (i.e,
$\sigma\in S_{n}$ acts by $\alpha_{j}^{i}\mapsto\alpha_{\sigma(j)}^{i}$). 
\begin{thm}
The cohomology ring $H^{*}(\X_{\Z^{r}}(\GL(n,\C));\Q)$ is isomorphic
to the invariant ring $\Lambda_{\Q}[\alpha_{1}^{1},...,\alpha_{n}^{r}]^{S_{n}}$. 
\end{thm}
\begin{sloppypar} 
\begin{proof}
As shown in \cite{FlLa}, $\X_{\Z^{r}}(\U(n))\cong T^{r}/W$ where
$W$ is the Weyl group of $\U(n)$ and $T\cong(S^{1})^{n}$ is a $n$-torus
(maximal torus in $\U(n)$). Therefore, $H^{*}(\X_{\Z^{r}}(\GL(n,\C));\Q)\cong H^{*}(T^{r}/W;\Q)$
by Theorem \ref{thm:Main} (see also Theorem \ref{pro:Sikora} and
Remark \ref{compactcase}). Since $W$ is finite and $T^{r}$ is compact,
results in \cite{Br} imply $H^{*}(T^{r}/W;\Q)\cong H^{*}(T^{r};\Q)^{W}$.
Then in \cite{Hatcher} it is shown that $H^{*}(T^{r};\Q)=H^{*}((S^{1})^{nr};\Q)$
is ring isomorphic to the exterior algebra $\Lambda_{\Q}[\alpha_{1}^{1},...,\alpha_{n}^{1},...,\alpha_{1}^{r},...,\alpha_{n}^{r}]$
where $\alpha_{j}^{i}$ can be explicitly described as the generators
of $H^{1}(S^{1};\Q)$ induced by $S^{1}\hookrightarrow T=(S^{1})^{n}\subset\U(n)\subset\GL(n,\C)\hookrightarrow\GL(n,\C)^{r}$,
with $i$ and $j$ labeling the inclusions into the corresponding
factors. Note that the Weyl group of $\U(n)$ (or $\GL(n,\C)$) is
isomorphic to $S_{n}$ and its action on $T^{r}$ is precisely the
action described above. Putting these facts together, we conclude
that the cohomology ring $H^{*}(\X_{\Z^{r}}(\GL(n,\C));\Q)$ is isomorphic
to the invariant ring $\Lambda_{\Q}[\alpha_{1}^{1},...,\alpha_{n}^{r}]^{S_{n}}$. 
\end{proof}
\end{sloppypar} Similar results hold for other groups $G$ in the
cases when $\X_{\Z^{r}}(G)$ is connected; that is, for $G=\SL(n,\C)$
and $\mathsf{Sp}(n,\C)$.

\section*{Acknowledgments}

The authors thank J. Souto and A. Pettet for generously sharing an
early version of their preprint \cite{PeSo} while we were generalizing
the proof of our theorem from \cite{FlLa}. The authors also thank
T. Baird for initially telling us about that preprint while the second
author was at the ICM in Hyderabad. We also thank A. Sikora for sharing
his preprint \cite{Si5} with us. The first named author thanks J. Dias
da Silva for motivating discussions on commuting matrices. The second named author thanks A. Sikora for discussions related to Remark
\ref{sikora} and T. Goodwillie for discussions related to Remark
\ref{goodwillie}; both on \url{mathoverflow.net}.  We also thank W. Goldman for conversations
related to Remark \ref{goldman}. Lastly, we thank the Institute Henri
Poincar\'e (IHP) and Centre de Recerca Matem\'atica (CRM) in Barcelona
for hosting us while this work was completed.

\appendix

\section{Proofs of Lemmata \ref{lem:delta-SLn} and \ref{lem:Appendix2}}

\label{appendix}

\subsection{\label{appendix1}Proof of Lemma \ref{lem:delta-SLn}.}

We first show that $\delta$ is well-defined. Let $h_{1}$ and $h_{2}$
be such that $d_{i}:=h_{i}gh_{i}^{-1}\in\Delta_{n}$. Then $(h_{1}h_{2}^{-1})d_{2}(h_{1}h_{2}^{-1})^{-1}=d_{1}$.
Two elements in a torus are conjugate if and only if they are conjugate
by an element of the Weyl group $W$ (see \cite{K}). Thus, $w:=h_{1}h_{2}^{-1}\in W$.
Recall the definition of $\delta_{t}$ in (\ref{eq:delta}) and that
$\delta_{t}|_{\Delta_{n}}$ coincides with $\sigma_{t}$ as defined
in (\ref{eq:sigma}) and so, it is $W$-equivariant for all $t$.
Thus: \begin{eqnarray*}
\delta_{t}(g) & = & h_{1}^{-1}\ \sigma_{t}(h_{1}gh_{1}^{-1})\ h_{1}=h_{1}^{-1}\ \sigma_{t}(d_{1})\ h_{1}\\
 & = & h_{1}^{-1}\ \sigma_{t}(w\ d_{2}\ w^{-1})\ h_{1}\\
 & = & h_{1}^{-1}w\ \sigma_{t}(d_{2})\ w^{-1}h_{1}\\
 & = & h_{2}^{-1}\ \sigma_{t}(h_{2}gh_{2}^{-1})\ h_{2}\end{eqnarray*}
 as required.

Secondly, $\delta_{t}(g)\in\SL(n,\C)_{ss}$ for all $t\in[0,1]$ and
$g\in\SL(n,\C)_{ss}$ since conjugates of elements in $\Delta_{n}$
are semisimple. Clearly, $\sigma$ is continuous, which implies $\delta$
is continuous as well.

Next, we show $\delta_{t}:\SL(n,\C)_{ss}\to\SL(n,\C)_{ss}$ is conjugation
equivariant. Let $l\in\SL(n,\C)$ and suppose that $hgh^{-1}\in\Delta_{n}$.
Then, $hl^{-1}$ diagonalizes $lgl^{-1}$ (as $(hl^{-1})lgl^{-1}(hl^{-1})^{-1}=hgh^{-1}$
is diagonal), and since $\delta$ is well-defined, we compute: \begin{eqnarray*}
\delta_{t}(lgl^{-1}) & = & (hl^{-1})^{-1}\sigma_{t}(hgh^{-1})hl^{-1}\\
 & = & l\ h^{-1}\sigma_{t}(hgh^{-1})h\ l^{-1}\\
 & = & l\ \delta_{t}(g)\ l^{-1},\end{eqnarray*}
 as wanted.

To prove that $\delta$ is a strong deformation retraction, it remains
to prove (a) for all $g\in\SL(n,\C)_{ss}$, $\delta_{0}(g)=g$ and
$\delta_{1}(g)=hkh^{-1}$ for some $k\in\SU(n)$ and $h\in\SL(n,\C)$,
and (b) for any $g\in\SL(n,\C)$ and $k\in\SU(n)$, then $\delta_{t}(gkg^{-1})=gkg^{-1}$
for all $t$.

Item (a) is immediate since the definition of $\delta_{t}$ immediately
implies $\delta_{0}$ is the identity on semisimple elements and that
$\delta_{1}$ maps into $\SL(n,\C)$-conjugates of $\SU(n)$. Now
we prove item (b). Note that $\delta_{t}$ is the identity on $\SU(n)$
since whenever $k=h^{-1}\mathrm{diag}(...,z_{j},...)h$ is unitary,
we have $|z_{j}|=1$ for all $1\leq j\leq n$ (showing $|z_{j}|^{-t}=1$
for all $t$). Thus, $\delta_{t}(gkg^{-1})=g\delta_{t}(k)g^{-1}=gkg^{-1}$
by equivariance of $\delta_{t}$. This completes the proof that $\delta$
is a strong deformation retraction.

Lastly, we prove $\delta_{t}$ preserves commutativity. Suppose $g_{1}g_{2}=g_{2}g_{1}$.
However, since they are diagonalizable, and are in $\SL(n,\C)$, we
know there is an element $h\in\SL(n,\C)$ which simultaneously triangulizes
them (see \cite{Dra}). However, since conjugation preserves semisimplicity
and non-diagonal upper triangular matrices are not semisimple, we
conclude that $h$ simultaneously diagonalizes $g_{1}$ and $g_{2}$.
Then $\delta_{t}(g_{1})$ and $\delta_{t}(g_{2})$ commute if and
only if $h\delta_{t}(g_{1})h^{-1}$ and $h\delta_{t}(g_{2})h^{-1}$
commute. It suffices to show that $\delta_{t}(hg_{1}h^{-1})$ and
$\delta_{t}(hg_{2}h^{-1})$ commute for all $t$ by equivariance.
However, both $\delta_{t}(hg_{1}h^{-1})$ and $\delta_{t}(hg_{2}h^{-1})$
are in $\Delta_{n}$ for all $t$, and thus they commute.

\subsection{\label{sub:A2}Proof of Lemma \ref{lem:Appendix2}}

Keeping the notation of Section \ref{sec:Four}, we start with a simple
lemma. 
\begin{lem}
\label{lem:power-map-subtorus}Let $H\cong\left(\mathbb{C}^{*}\right)^{k}$
be a subtorus of $\Delta_{n}\cong\left(\mathbb{C}^{*}\right)^{n-1}$
$($so $k\leq n-1)$. Then $p_{m}^{-1}(H)\subset\Delta_{n}$ is the
disjoint union of a finite number of affine subvarieties, all of which
are isomorphic to $H$ $($as varieties, but not as groups$)$. More
precisely, every connected component of $p_{m}^{-1}(H)$ is a coset
$gH$, for some $g\in p_{m}^{-1}(H)$.\end{lem}
\begin{proof}
$gH$ are all in $p_{m}^{-1}(H)$ for all $g$ in $p_{m}^{-1}(H)$
since the inverse image is contained in $\Delta_{n}$ (it is Abelian).
In fact, $p_{m}^{-1}(H)$ is a subgroup for this reason. Therefore
it is reductive since $p_{m}$ is algebraic (and thus continuous),
and inverse images of Zariski closed sets under algebraic maps are
Zariski closed. However, Zariski closed subgroups of a torus are reductive
since the radical must be a torus, and therefore $p_{m}^{-1}(H)$
is a reductive subgroup and so has finitely many components. Since
cosets are disjoint and homeomophic, each is closed in $\Delta_{n}$
since $H$ is closed in $\Delta_{n}$. On the other hand, each is
open in the subspace topology and so are components. Only the identity
coset is a group in and of itself, the rest are algebraic subsets
(closed cosets are algebraic). 
\end{proof}
Now we conclude the proof of Lemma \ref{lem:Appendix2}. 
\begin{proof}
(1) Let $g\in G_{ss}$ and assume $g\in G_{f}$ with $f^{m}=\mathbf{1}$.
Then, as in Lemma \ref{lem:semisimple-inclusion}, there is $h\in G_{0}$
so that $hg^{m}h^{-1}\in T_{0}\subset\Delta_{n}$. By definition of
$\delta$ and property (6) of Lemma \ref{lem:delta-SLn}, we have
$(\delta_{t}(g))^{m}=\delta_{t}(g^{m})=h^{-1}\sigma_{t}(hg^{m}h^{-1})h$.
So, \begin{equation}
\sigma_{t}(hg^{m}h^{-1})=h(\delta_{t}(g))^{m}h^{-1}=(h\delta_{t}(g)h^{-1})^{m}.\label{lemeqn}\end{equation}

Next, we prove $\sigma_{t}$ preserves $T_{0}$ for all $t\in[0,1]$.
Indeed, this deformation is nothing but the usual polar deformation
restricted to the torus $T_{0}$, which therefore stays in $T_{0}$;
see page 117 in \cite{K}. The main idea is this: the torus $T_{0}$
is algebraically cut out of the diagonals $\Delta_{n}$, however since
$T_{0}$ is a group and $\mathrm{diag}(e^{i\theta_{1}},...,e^{i\theta_{n}})\in K_{0}$
we know that all positive multiples of $\mathrm{diag}(r_{1},...,r_{n})$
are in $T_{0}$ which implies all positive multiples satisfy the relations.
This then implies that $\mathrm{diag}(r_{1}^{1-t},...,r_{n}^{1-t})$
satisfy the relations for all $t$, which in turn implies $\mathrm{diag}(r_{1}^{1-t}e^{i\theta_{1}},...,r_{n}^{1-t}e^{i\theta_{n}})$
is in $T_{0}$ for all $t$.

Now, we use the fact that $\sigma_{t}$ preserves $T_{0}$ for all
$t\in[0,1]$ to conclude (from Equation \ref{lemeqn}) that for all
$t$, $(h\delta_{t}(g)h^{-1})^{m}\in T_{0}$, and therefore $h\delta_{t}(g)h^{-1}\in p_{m}^{-1}(T_{0})$
for all $t$. By continuity of $\delta_{t}$ with respect to the parameter
$t$, we conclude that $h\delta_{t}(g)h^{-1}$ is always in the same
connected component of $p_{m}^{-1}(T_{0})$ for all $t$; denote this
component by $T_{1}$. Now, by Lemma \ref{lem:power-map-subtorus}
every connected component of $p_{m}^{-1}(T_{0})$ is of the form $\tilde{g}T_{0}$
for some $\tilde{g}\in p_{m}^{-1}(T_{0})$.

Since $hgh^{-1}=h\delta_{0}(g)h^{-1}\in T_{1}\subset p_{m}^{-1}(T_{0})$,
we may take $\tilde{g}=hgh^{-1}\in G$. Therefore, the coset $T_{1}=\tilde{g}T_{0}=hgh^{-1}T_{0}\subset G$.
More generally, this means that any connected component of $p_{m}^{-1}(T_{0})$
which contains a point of $G$, is completely contained inside $G$.
So, we conclude that $h\delta_{t}(g)h^{-1}\in G$, which means that
$\delta_{t}(g)\in G$ for all $t$.

Finally $\delta_{t}(g)\in G_{ss}$ since it is a $m^{{\rm th}}$ root
of a semisimple $\delta_{t}(g^{m})$, and $g_{u}^{m}=\mathbf{1}$
implies $g_{u}=\mathbf{1}$ (\cite{Bo}, page 87).

(2) If we let $t=1$ in the formulas above, since conjugates of semisimple
elements are semisimple (\cite{Bo}, page 85), we get $h\delta_{1}(g)h^{-1}\in G_{ss}\cap\SU(n)\subset K$,
which means by definition, that $\delta_{1}(g)\in G(K)$. 
\end{proof}


\def\cdprime{$''$} \def\cdprime{$''$} \def\cprime{$'$} \def\cprime{$'$}
  \def\cprime{$'$} \def\cprime{$'$} \def\cprime{$'$}

\end{document}